\documentclass{lmcs}
\pdfoutput=1

\usepackage{lastpage}
\lmcsdoi{17}{4}{20}
\lmcsheading{}{\pageref{LastPage}}{}{}%
{Feb.~26,~2021}{Dec.~17,~2021}{}

\usepackage{amsmath}
\usepackage{amssymb}
\usepackage{amsthm}

\usepackage{tikz}

\usepackage[utf8]{inputenc}

\def\PA{\mathsf{PA}}
\def\HYP{\mathsf{HYP}}

\def\NON{\mathsf{NON}}

\def\NRNG{\mathsf{NRNG}}
\def\NMAJ{\mathsf{NMAJ}}
\def\BCT{\mathsf{BCT}}

\def\WGEN{\mathsf{WGEN}}
\def\NLIM{\mathsf{NLIM}}
\def\EXT{\mathsf{EXT}}
\def\RET{\mathsf{RET}}
\def\MEET{\mathsf{MEET}}
\def\NLOW{\mathsf{NLOW}}

\def\NGEQ{\mathsf{NGEQ}}

\def\DIS{\mathsf{DIS}}

\def\EC{\mathsf{EC}}

\def\deg{\mathrm{deg}}

\newcommand{\uwidehat}[1]{%
  \mathpalette\douwidehat{#1}%
}
\makeatletter
\newcommand{\douwidehat}[2]{%
  \sbox0{$\m@th#1\widehat{\hphantom{#2}}$}%
  \sbox2{$\m@th#1x$}
  \sbox4{$\m@th#1#2$}
  \dimen0=\ht0
  \advance\dimen0 -.8\ht2
  \dimen2=\dp4
  \rlap{%
    \raisebox{\dimexpr\dimen0-\dimen2}{%
      \scalebox{1}[-1]{\box0}%
    }%
  }%
  {#2}%
}
\makeatother

\usepackage{catchfile}

\usepackage{lineno}

\title{Stashing And Parallelization Pentagons}

\author[V.~Brattka]{Vasco Brattka}
\address{Faculty of Computer Science, Universit\"at der Bundeswehr M\"unchen, Germany and
Department of Mathematics and Applied Mathematics, University of Cape Town, South Africa}
\email{Vasco.Brattka@cca-net.de}

\begin{document} 



\def\AA{{\mathcal A}}
\def\BB{{\mathcal B}}
\def\CC{{\mathcal C}}
\def\DD{{\mathcal D}}
\def\EE{{\mathcal E}}
\def\FF{{\mathcal F}}
\def\GG{{\mathcal G}}
\def\HH{{\mathcal H}}
\def\II{{\mathcal I}}
\def\JJ{{\mathcal J}}
\def\KK{{\mathcal K}}
\def\LL{{\mathcal L}}
\def\MM{{\mathcal M}}
\def\NN{{\mathcal N}}
\def\OO{{\mathcal O}}
\def\PP{{\mathcal P}}
\def\QQ{{\mathcal Q}}
\def\RR{{\mathcal R}}
\def\SS{{\mathcal S}}
\def\TT{{\mathcal T}}
\def\UU{{\mathcal U}}
\def\VV{{\mathcal V}}
\def\WW{{\mathcal W}}
\def\XX{{\mathcal X}}
\def\YY{{\mathcal Y}}
\def\ZZ{{\mathcal Z}}


\def\bA{{\mathbf A}}
\def\bB{{\mathbf B}}
\def\bC{{\mathbf C}}
\def\bD{{\mathbf D}}
\def\bE{{\mathbf E}}
\def\bF{{\mathbf F}}
\def\bG{{\mathbf G}}
\def\bH{{\mathbf H}}
\def\bI{{\mathbf I}}
\def\bJ{{\mathbf J}}
\def\bK{{\mathbf K}}
\def\bL{{\mathbf L}}
\def\bM{{\mathbf M}}
\def\bN{{\mathbf N}}
\def\bO{{\mathbf O}}
\def\bP{{\mathbf P}}
\def\bQ{{\mathbf Q}}
\def\bR{{\mathbf R}}
\def\bS{{\mathbf S}}
\def\bT{{\mathbf T}}
\def\bU{{\mathbf U}}
\def\bV{{\mathbf V}}
\def\bW{{\mathbf W}}
\def\bX{{\mathbf X}}
\def\bY{{\mathbf Y}}
\def\bZ{{\mathbf Z}}


\def\IB{{\Bbb{B}}}
\def\IC{{\Bbb{C}}}
\def\IF{{\Bbb{F}}}
\def\IN{{\Bbb{N}}}
\def\IP{{\Bbb{P}}}
\def\IQ{{\Bbb{Q}}}
\def\IR{{\Bbb{R}}}
\def\IS{{\Bbb{S}}}
\def\IT{{\Bbb{T}}}
\def\IZ{{\Bbb{Z}}}

\def\IIB{{\Bbb{\mathbf B}}}
\def\IIC{{\Bbb{\mathbf C}}}
\def\IIN{{\Bbb{\mathbf N}}}
\def\IIQ{{\Bbb{\mathbf Q}}}
\def\IIR{{\Bbb{\mathbf R}}}
\def\IIZ{{\Bbb{\mathbf Z}}}


\def\ELSE{\quad\mbox{else}\quad}
\def\WITH{\quad\mbox{with}\quad}
\def\FOR{\quad\mbox{for}\quad}
\def\AND{\;\mbox{and}\;}
\def\OR{\;\mbox{or}\;}

\def\To{\longrightarrow}
\def\TO{\Longrightarrow}
\def\In{\subseteq}
\def\sm{\setminus}
\def\Inneq{\In_{\!\!\!\!/}}
\def\dmin{\mathop{\dot{-}}}
\def\splus{\oplus}
\def\SEQ{\triangle}
\def\DIV{\uparrow}
\def\INV{\leftrightarrow}
\def\SET{\Diamond}

\def\kto{\equiv\!\equiv\!>}
\def\kin{\subset\!\subset}
\def\pto{\leadsto}
\def\into{\hookrightarrow}
\def\onto{\to\!\!\!\!\!\to}
\def\prefix{\sqsubseteq}
\def\rel{\leftrightarrow}
\def\mto{\rightrightarrows}

\def\B{{\mathsf{{B}}}}
\def\D{{\mathsf{{D}}}}
\def\G{{\mathsf{{G}}}}
\def\E{{\mathsf{{E}}}}
\def\J{{\mathsf{{J}}}}
\def\K{{\mathsf{{K}}}}
\def\L{{\mathsf{{L}}}}
\def\R{{\mathsf{{R}}}}
\def\T{{\mathsf{{T}}}}
\def\U{{\mathsf{{U}}}}
\def\W{{\mathsf{{W}}}}
\def\Z{{\mathsf{{Z}}}}
\def\w{{\mathsf{{w}}}}
\def\HP{{\mathsf{{H}}}}
\def\C{{\mathsf{{C}}}}
\def\Tot{{\mathsf{{Tot}}}}
\def\Fin{{\mathsf{{Fin}}}}
\def\Cof{{\mathsf{{Cof}}}}
\def\Cor{{\mathsf{{Cor}}}}
\def\Equ{{\mathsf{{Equ}}}}
\def\Com{{\mathsf{{Com}}}}
\def\Inf{{\mathsf{{Inf}}}}

\def\Tr{{\mathrm{Tr}}}
\def\Sierp{{\mathrm Sierpi{\'n}ski}}
\def\psisierp{{\psi^{\mbox{\scriptsize\Sierp}}}}
\def\cl{{\mathrm{{cl}}}}
\def\Haus{{\mathrm{{H}}}}
\def\Ls{{\mathrm{{Ls}}}}
\def\Li{{\mathrm{{Li}}}}

\def\CL{\mathsf{CL}}
\def\ACC{\mathsf{ACC}}
\def\AoUC{\mathsf{AUC}}
\def\DNC{\mathsf{DNC}}
\def\ATR{\mathsf{ATR}}
\def\LPO{\mathsf{LPO}}
\def\LLPO{\mathsf{LLPO}}
\def\WKL{\mathsf{WKL}}
\def\RCA{\mathsf{RCA}}
\def\ACA{\mathsf{ACA}}
\def\SEP{\mathsf{SEP}}
\def\BCT{\mathsf{BCT}}
\def\IVT{\mathsf{IVT}}
\def\IMT{\mathsf{IMT}}
\def\OMT{\mathsf{OMT}}
\def\CGT{\mathsf{CGT}}
\def\UBT{\mathsf{UBT}}
\def\BWT{\mathsf{BWT}}
\def\HBT{\mathsf{HBT}}
\def\BFT{\mathsf{BFT}}
\def\FPT{\mathsf{FPT}}
\def\WAT{\mathsf{WAT}}
\def\LIN{\mathsf{LIN}}
\def\B{\mathsf{B}}
\def\BF{\mathsf{B_\mathsf{F}}}
\def\BI{\mathsf{B_\mathsf{I}}}
\def\C{\mathsf{C}}
\def\CF{\mathsf{C_\mathsf{F}}}
\def\CN{\mathsf{C_{\IN}}}
\def\CI{\mathsf{C_\mathsf{I}}}
\def\CK{\mathsf{C_\mathsf{K}}}
\def\CA{\mathsf{C_\mathsf{A}}}
\def\WPO{\mathsf{WPO}}
\def\WLPO{\mathsf{WLPO}}
\def\MP{\mathsf{MP}}
\def\BD{\mathsf{BD}}
\def\Fix{\mathsf{Fix}}
\def\Mod{\mathsf{Mod}}

\newcommand{\hop}{{\mbox{$\boldsymbol{\cdot}$}}}

\def\w{\mathsf{w}}

\def\leqm{\mathop{\leq_{\mathrm{m}}}}
\def\equivm{\mathop{\equiv_{\mathrm{m}}}}
\def\leqT{\mathop{\leq_{\mathrm{T}}}}
\def\lT{\mathop{<_{\mathrm{T}}}}
\def\nleqT{\mathop{\not\leq_{\mathrm{T}}}}
\def\equivT{\mathop{\equiv_{\mathrm{T}}}}
\def\nequivT{\mathop{\not\equiv_{\mathrm{T}}}}
\def\leqwtt{\mathop{\leq_{\mathrm{wtt}}}}
\def\equiPT{\mathop{\equiv_{\P\mathrm{T}}}}
\def\leqW{\mathop{\leq_{\mathrm{W}}}}
\def\equivW{\mathop{\equiv_{\mathrm{W}}}}
\def\leqtW{\mathop{\leq_{\mathrm{tW}}}}
\def\leqSW{\mathop{\leq_{\mathrm{sW}}}}
\def\equivSW{\mathop{\equiv_{\mathrm{sW}}}}
\def\leqPW{\mathop{\leq_{\widehat{\mathrm{W}}}}}
\def\equivPW{\mathop{\equiv_{\widehat{\mathrm{W}}}}}
\def\leqFPW{\mathop{\leq_{\mathrm{W}^*}}}
\def\equivFPW{\mathop{\equiv_{\mathrm{W}^*}}}
\def\leqWW{\mathop{\leq_{\overline{\mathrm{W}}}}}
\def\nleqW{\mathop{\not\leq_{\mathrm{W}}}}
\def\nleqSW{\mathop{\not\leq_{\mathrm{sW}}}}
\def\lW{\mathop{<_{\mathrm{W}}}}
\def\lSW{\mathop{<_{\mathrm{sW}}}}
\def\nW{\mathop{|_{\mathrm{W}}}}
\def\nSW{\mathop{|_{\mathrm{sW}}}}
\def\leqt{\mathop{\leq_{\mathrm{t}}}}
\def\equivt{\mathop{\equiv_{\mathrm{t}}}}
\def\leqtop{\mathop{\leq_\mathrm{t}}}
\def\equivtop{\mathop{\equiv_\mathrm{t}}}

\def\bigtimes{\mathop{\mathsf{X}}}

\def\leqm{\mathop{\leq_{\mathrm{m}}}}
\def\equivm{\mathop{\equiv_{\mathrm{m}}}}
\def\leqT{\mathop{\leq_{\mathrm{T}}}}
\def\leqM{\mathop{\leq_{\mathrm{M}}}}
\def\equivT{\mathop{\equiv_{\mathrm{T}}}}
\def\equiPT{\mathop{\equiv_{\P\mathrm{T}}}}
\def\leqW{\mathop{\leq_{\mathrm{W}}}}
\def\equivW{\mathop{\equiv_{\mathrm{W}}}}
\def\nequivW{\mathop{\not\equiv_{\mathrm{W}}}}
\def\leqSW{\mathop{\leq_{\mathrm{sW}}}}
\def\equivSW{\mathop{\equiv_{\mathrm{sW}}}}
\def\leqPW{\mathop{\leq_{\widehat{\mathrm{W}}}}}
\def\equivPW{\mathop{\equiv_{\widehat{\mathrm{W}}}}}
\def\nleqW{\mathop{\not\leq_{\mathrm{W}}}}
\def\nleqSW{\mathop{\not\leq_{\mathrm{sW}}}}
\def\lW{\mathop{<_{\mathrm{W}}}}
\def\lSW{\mathop{<_{\mathrm{sW}}}}
\def\nW{\mathop{|_{\mathrm{W}}}}
\def\nSW{\mathop{|_{\mathrm{sW}}}}

\def\botW{\mathbf{0}}
\def\midW{\mathbf{1}}
\def\topW{\mathbf{\infty}}

\def\pol{{\leq_{\mathrm{pol}}}}
\def\rem{{\mathop{\mathrm{rm}}}}

\def\cc{{\mathrm{c}}}
\def\d{{\,\mathrm{d}}}
\def\e{{\mathrm{e}}}
\def\ii{{\mathrm{i}}}

\def\Cf{C\!f}
\def\id{{\mathrm{id}}}
\def\pr{{\mathrm{pr}}}
\def\inj{{\mathrm{inj}}}
\def\cf{{\mathrm{cf}}}
\def\dom{{\mathrm{dom}}}
\def\range{{\mathrm{range}}}
\def\graph{{\mathrm{graph}}}
\def\Graph{{\mathrm{Graph}}}
\def\epi{{\mathrm{epi}}}
\def\hypo{{\mathrm{hypo}}}
\def\Lim{{\mathrm{Lim}}}
\def\diam{{\mathrm{diam}}}
\def\dist{{\mathrm{dist}}}
\def\supp{{\mathrm{supp}}}
\def\union{{\mathrm{union}}}
\def\fiber{{\mathrm{fiber}}}
\def\ev{{\mathrm{ev}}}
\def\mod{{\mathrm{mod}}}
\def\sat{{\mathrm{sat}}}
\def\seq{{\mathrm{seq}}}
\def\lev{{\mathrm{lev}}}
\def\mind{{\mathrm{mind}}}
\def\arccot{{\mathrm{arccot}}}

\def\Add{{\mathrm{Add}}}
\def\Mul{{\mathrm{Mul}}}
\def\SMul{{\mathrm{SMul}}}
\def\Neg{{\mathrm{Neg}}}
\def\Inv{{\mathrm{Inv}}}
\def\Ord{{\mathrm{Ord}}}
\def\Sqrt{{\mathrm{Sqrt}}}
\def\Re{{\mathrm{Re}}}
\def\Im{{\mathrm{Im}}}
\def\Sup{{\mathrm{Sup}}}

\def\LSC{{\mathcal LSC}}
\def\USC{{\mathcal USC}}

\def\CE{{\mathcal{E}}}
\def\Pref{{\mathrm{Pref}}}

\def\Baire{\IN^\IN}

\def\TRUE{{\mathrm{TRUE}}}
\def\FALSE{{\mathrm{FALSE}}}

\def\co{{\mathrm{co}}}

\def\BBB{{\tt B}}

\newcommand{\SO}[1]{{{\mathbf\Sigma}^0_{#1}}}
\newcommand{\SI}[1]{{{\mathbf\Sigma}^1_{#1}}}
\newcommand{\PO}[1]{{{\mathbf\Pi}^0_{#1}}}
\newcommand{\PI}[1]{{{\mathbf\Pi}^1_{#1}}}
\newcommand{\DO}[1]{{{\mathbf\Delta}^0_{#1}}}
\newcommand{\DI}[1]{{{\mathbf\Delta}^1_{#1}}}
\newcommand{\sO}[1]{{\Sigma^0_{#1}}}
\newcommand{\sI}[1]{{\Sigma^1_{#1}}}
\newcommand{\pO}[1]{{\Pi^0_{#1}}}
\newcommand{\pI}[1]{{\Pi^1_{#1}}}
\newcommand{\dO}[1]{{{\Delta}^0_{#1}}}
\newcommand{\dI}[1]{{{\Delta}^1_{#1}}}
\newcommand{\sP}[1]{{\Sigma^\P_{#1}}}
\newcommand{\pP}[1]{{\Pi^\P_{#1}}}
\newcommand{\dP}[1]{{{\Delta}^\P_{#1}}}
\newcommand{\sE}[1]{{\Sigma^{-1}_{#1}}}
\newcommand{\pE}[1]{{\Pi^{-1}_{#1}}}
\newcommand{\dE}[1]{{\Delta^{-1}_{#1}}}

\newcommand{\dBar}[1]{{\overline{\overline{#1}}}}

\def\QED{$\hspace*{\fill}\Box$}
\def\rand#1{\marginpar{\rule[-#1 mm]{1mm}{#1mm}}}

\def\BL{\BB}


\newcommand{\bra}[1]{\langle#1|}
\newcommand{\ket}[1]{|#1\rangle}
\newcommand{\braket}[2]{\langle#1|#2\rangle}

\newcommand{\ind}[1]{{\em #1}\index{#1}}
\newcommand{\mathbox}[1]{\[\fbox{\rule[-4mm]{0cm}{1cm}$\quad#1$\quad}\]}


\newenvironment{eqcase}{\left\{\begin{array}{lcl}}{\end{array}\right.}


\begin{abstract}
Parallelization is an algebraic operation that lifts problems to sequences in a natural way.
Given a sequence as an instance of the parallelized problem, 
another sequence is a solution of this problem if {\em every} component is instance-wise a solution of the original problem.
In the Weihrauch lattice parallelization is a closure operator that corresponds to the bang operator in linear logic.
Here we introduce a dual operation that we call {\em stashing} and that also
lifts problems to sequences, but such that only {\em some} component has to be an instance-wise solution.
In this case the solution is stashed away in the sequence.
This operation, if properly defined, induces an interior operator in the Weihrauch lattice, which
corresponds to the question mark operator known from linear logic. 
It can also be seen as a countable version of the sum operation.
We also study the action of the monoid induced by stashing and parallelization on the Weihrauch lattice,
and we prove that it leads to at most five distinct degrees, which (in the maximal case) are always
organized in pentagons. 
We also introduce another closely related interior operator in the Weihrauch lattice that replaces solutions
of problems by upper Turing cones that are strong enough to compute solutions. 
It turns out that on parallelizable degrees this interior operator corresponds to stashing. 
This implies that, somewhat surprisingly, all problems which are simultaneously parallelizable and stashable 
have computability-theoretic characterizations. 
Finally, we apply all these results in order to study the recently introduced discontinuity problem, 
which appears as the bottom of a number of natural stashing-parallelization pentagons.
The discontinuity problem is not only the stashing of several variants of the lesser limited principle of omniscience,
but it also parallelizes to the non-computability problem. This supports the slogan that
``non-computability is the parallelization of discontinuity''. We also study the non-majorization problem
as an asymmetric version of the discontinuity problem and we show that it parallelizes to the hyperimmunity problem.
Finally we identify a phase transition related to the limit avoidance problem that marks a point where pentagons are taking off from the bottom of the Weihrauch lattice.
\end{abstract}

\maketitle


\section{Introduction}

The Weihrauch lattice has been used as a computability theoretic framework to analyze the uniform computational
content of mathematical problems from many different areas of mathematics, and it can also be seen as a uniform variant
of reverse mathematics (a recent survey on Weihrauch complexity can be found in \cite{BGP21}). 

The notion of a mathematical problem has a very general definition in this approach.

\begin{defi}[Problems]
\label{def:problem}
A {\em problem} is a multi-valued function $f:\In X\mto Y$ on represented spaces
$X,Y$ that has a realizer.
\end{defi}

We recall that by a {\em realizer} $F:\In\IN^\IN\to\IN^\IN$ of $f$, we mean a function $F$ that satisfies
$\delta_YF(p)\in f\delta_X(p)$ for all $p\in\dom(f\delta_X)$, where $\delta_X:\In\IN^\IN\to X$ and ${\delta_Y:\In\IN^\IN\to Y}$
are the representations of $X$ and $Y$, respectively (i.e., partial surjective maps onto $X$ and $Y$, respectively).
A problem is called {\em computable} if it has a computable realizer and {\em continuous} if it has a continuous realizer.

By $\langle p,q\rangle$ we denote the usual
pairing function on $\IN^\IN$, defined by $\langle p,q\rangle(2n)=p(n)$, $\langle p,q\rangle(2n+1)=q(n)$
for all $p,q\in\IN^\IN,n\in\IN$.
Weihrauch reducibility can now be defined as follows. 

\begin{defi}[Weihrauch reducibility]
Let $f:\In X\mto Y$ and $g:\In W\mto Z$ be problems. Then $f$ is called {\em Weihrauch reducible}
to $g$, in symbols $f\leqW g$, if there are computable $H,K:\In\IN^\IN\to\IN^\IN$ such that
$H\langle \id,GK\rangle$ is a realizer of $f$ whenever $G$ is a realizer of $g$.
Analogously, one says that $f$ is {\em strongly Weihrauch reducible} to $g$, in symbols $f\leqSW g$, if the expression $H\langle \id,GK\rangle$
can be replaced by $HGK$. Both versions of the reducibility have topological counterparts,
where one only requires $H,K$ to be continuous and these reducibilities are denoted by $\leq_\mathrm{W}^*$ and $\leq_\mathrm{sW}^*$,
respectively.
\end{defi}

The topological version of Weihrauch reducibility has always been studied alongside the computability-theoretic
version, and all four reducibilities induce a lattice structure (see \cite{BGP21} for references). 

Normally, the {\em Weihrauch lattice} refers
to the lattice induced by $\leqW$, but here we will freely use this term also for the lattice structure induced by $\leq_\mathrm{W}^*$. If we want to be more precise, we will call the latter the {\em topological Weihrauch lattice}.
The equivalence classes of problems under (strong) Weihrauch reducibility are called {\em (strong) Weihrauch degrees}.

In \cite[Definition~4.1]{BG11} the operation of parallelization was introduced. For reasons that will become clear below, we
denote the parallelization $\widehat{f}$ of a problem $f$ additionally with the non-standard notation $\Pi f$ in this article.

\begin{defi}[Parallelization]
For every problem $f:\In X\mto Y$ we define its {\em parallelization} 
$\Pi f:\In X^\IN\mto Y^\IN$ by $\dom(\Pi f):=\dom(f)^\IN$ and
\[\Pi f(x_n):=\{(y_n)\in Y^\IN:(\forall n)\; y_n\in f(x_n)\}\]
for all $(x_n)\in\dom(\Pi f)$. We also write $\widehat{f}:=\Pi f$ and we call a problem {\em parallelizable}
(or {\em strongly parallelizable}) if $f\equivW\widehat{f}$ (or $f\equivSW\widehat{f}$) holds.
\end{defi}

In \cite[Proposition~4.2]{BG11} it was proved that parallelization is a closure operator in the Weihrauch lattice.
This holds analogously for the topological versions of Weihrauch reducibility.

\begin{fact}[Parallelization]
\label{fact:parallelization} 
$f\mapsto\Pi f$ is a closure operator with respect to the following versions of Weihrauch reducibility: $\leqW,\leqSW,\leq_\mathrm{W}^*$ and $\leq_\mathrm{sW}^*$.
\end{fact}

We recall the definition of a closure operator and an interior operator for a preordered set. By
a {\em preordered set} $(P,\leq)$ we mean a set $P$ with a relation $\leq$ on $P$ that is reflexive and transitive.
The relations $\leqW,\leqSW,\leq_\mathrm{W}^*$ and $\leq_\mathrm{sW}^*$ are preorders 
on the set $P$ of problems $f:\In\IN^\IN\mto\IN^\IN$ on Baire space\footnote{For studying the order structure it is sufficient
to consider problems on Baire space as representatives of arbitrary problems. This guarantees that $P$ is actually a set.}.

\begin{defi}[Closure and interior operator]
Let $(P,\leq)$ a preordered set with a function $C:P\to P$. Then $C$ is called a {\em closure operator}
for $\leq$ if the following hold for all $x,y\in P$:
\begin{enumerate}
\item $x\leq C(x)$ \hfill (extensive)
\item $x\leq y\TO C(x)\leq C(y)$ \hfill(monotone)
\item $CC(x)\leq C(x)$ \hfill(idempotent)
\end{enumerate}
Analogously, $C$ is called an {\em interior operator} for $\leq$ if the three conditions
hold for $\geq$ in place of $\leq$.
\end{defi}

Besides $f\mapsto\Pi f$ a number of other closure operator appeared in the study of the Weihrauch lattice~\cite{BGP21}.
However, not so many interior operators have been considered yet.
In this article we want to study a dual operation to parallelization that we call {\em stashing} and that
can be defined as follows. 

\begin{defi}[Stashing]
\label{def:stashing}
For every problem $f:\In X\mto Y$ we define its {\em stashing} or {\em summation}
$\Sigma f:\In X^\IN\mto \overline{Y}^\IN$ by $\dom(\Sigma f):=\dom(f)^\IN$ and
\[\Sigma f(x_n):=\{(y_n)\in\overline{Y}^\IN:(\exists n)\; y_n\in f(x_n)\}\]
for all $(x_n)\in\dom(\Sigma f)$. We also write $\uwidehat{f}:=\Sigma f$ and we call a problem {\em stashable}
(or {\em strongly stashable}) if $f\equivW\uwidehat{f}$ (or $f\equivSW\uwidehat{f}$) holds.
\end{defi}

Essentially, the definition corresponds to parallelization with an existential quantifier in the place of the universal one.
This means that given an instance $(x_n)$ for $\Sigma f$, a solution is a sequence $(y_n)$ such that 
$y_n\in f(x_n)$ for at least one $n\in\IN$. This operation can be seen as a countable version of the sum operation $+$ (see \cite{BGP21}),
which is the reason why we have called it {\em summation} in earlier presentations of this work. 
However, {\em stashing} better corresponds to the intuition of what $\Sigma f$ does.
 
A subtle technical point in this definition is that we use the {\em completion} $\overline{Y}$ of the space
$Y$ on the output side. For a represented space $(Y,\delta_Y)$ the {\em completion} $(\overline{Y},\delta_{\overline{Y}})$
is defined by $\overline{Y}:=Y\cup\{\bot\}$ (with a distinct element $\bot\not\in Y$) and $\delta_{\overline{Y}}:\IN^\IN\to\overline{Y}$ with
\[\delta_{\overline{Y}}(p):=\left\{\begin{array}{ll}
 \delta_Y(p-1) & \mbox{if $p-1\in\dom(\delta_Y)$}\\
  \bot & \mbox{otherwise}
\end{array}\right.,
\]
where $p-1\in\IN^\IN\cup\IN^*$ is a finite or infinite sequence that is obtained as the concatenation of
$p(0)-1,p(1)-1,p(2)-1,...$ with the understanding that $-1=\varepsilon\in\IN^*$ is the empty word.
This operation of completion saw some recent surge of interest after Dzhafarov~\cite{Dzh19} used it to show
that the strong version Weihrauch reducibility $\leqSW$ actually yields a lattice structure (here completion appeared
in the definition of a suitable supremum operation). See \cite{BG20,BG21a} for further applications and
a more detailed study of completion. 

One reason that the completion cannot be omitted in Definition~\ref{def:stashing} is that it allows
us to produce dummy outputs with no meaning (without the completion this might not be possible
as, for instance, some represented spaces $(Y,\delta_Y)$ might not even have computable points).
Another reason is that every partial computable problem with a completion on the output side
can be extended to a total computable problem in a certain sense.
In any case the completion enables us to prove the following result (see Proposition~\ref{prop:stashing})
in Section~\ref{sec:stashing}.

\begin{prop}[Stashing]
$f\mapsto\Sigma f$ is an interior operator with respect to the following versions of Weihrauch reducibility: $\leqW,\leqSW,\leq_\mathrm{W}^*$ and $\leq_\mathrm{sW}^*$.
\end{prop}

While parallelization $\Pi$ can be seen as the counterpart of the 
bang operator ``!' in linear logic \cite{BGP21}, stashing $\Sigma$ can be seen as the counterpart of the dual why-not operator ``?''.

Since our lattice is now equipped with a closure operator $f\mapsto\Pi f$ and a dual interior operator
$f\mapsto \Sigma f$, it is natural to ask how the monoid generated by $\{\Pi,\Sigma\}^*$ acts on 
the lattice structure? In other words, starting from an arbitrary problem $f$, what kind of problems can 
we generate by repeated applications of $\Pi$ and $\Sigma$ (in any order) to $f$? 
And how are these problems related with respect to the lattice structure?

In Section~\ref{sec:monoid} we prove that we can generate at most five distinct problems in this way (up to equivalence) and that
these five problems (in the maximal case) are always organized in a pentagon (see Proposition~\ref{prop:monoid}, Corollary~\ref{cor:full-pentagon}).

\def\colorf{orange!30}
\def\colorpf{red!30}
\def\colorsf{blue!30}
\def\colorspf{purple!30}
\def\colorpsf{violet!30}

\begin{figure}
\begin{tikzpicture}[scale=2]
\node[style={fill=\colorf},left] at (-1.03,0) {$\ACC_\IN$};
\node[fill,circle,scale=0.3] (f) at (-1,0) {};
\node[style={fill=\colorsf},below] at (-0.3,-0.981) {$\DIS$};
\node[fill,circle,scale=0.3] (Sf) at (-0.3,-0.951) {};
\node[style={fill=\colorpf},above] at (-0.3,0.981) {$\DNC_\IN$};
\node[fill,circle,scale=0.3]  (Pf) at (-0.3,0.951) {};
\node[style={fill=\colorspf},right] at  (0.83,0.588)  {$\DNC_\IN^\DD$};
\node[fill,circle,scale=0.3] (SPf) at (0.8,0.588) {};
\node[style={fill=\colorpsf},right] at (0.83,-0.588) {$\NON$};
\node[fill,circle,scale=0.3]  (PSf) at (0.8,-0.588) {};
\draw[thick] (Pf) -- (f);
\draw[thick] (Pf) -- (SPf);
\draw[thick] (f) -- (Sf);
\draw[thick] (SPf) -- (PSf);
\draw[thick] (PSf) -- (Sf);
\node at (-0.75,0.585) {$\Pi$};
\node at (0.38,0.88) {$\Sigma$};
\node at (-0.75,-0.585) {$\Sigma$};
\node at (0.38,-0.88) {$\Pi$};
\end{tikzpicture}
\caption{$\ACC_\IN$ pentagon in the Weihrauch lattice.} 
\label{fig:ACCN-pentagon}
\end{figure}
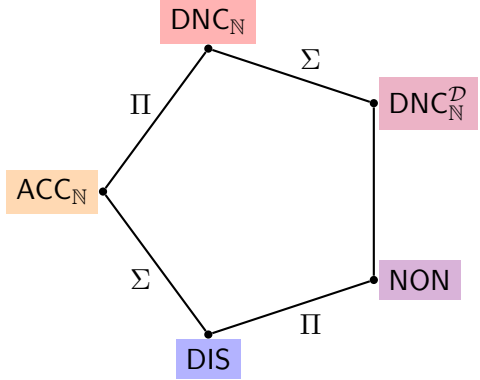

The maximal case can actually occur and in Section~\ref{sec:discontinuity-problem} we study a number
of specific such pentagons, in particular the one shown in the diagram in Figure~\ref{fig:ACCN-pentagon}.
Here every problem in the diagram allows a $\leqSW$--reduction to any problem above it that is connected with a line
and no other $\leq_\mathrm{W}^*$--reductions are possible (except for those that follow by transitivity).
We define all the problems that occur in this diagram and some further problems that we are going to study in this article.

\begin{defi}[Some problems]
\label{def:problems}
We consider the following problems for $X\In\IN$:
\begin{enumerate}
\item $\LPO:\IN^\IN\to\{0,1\},\LPO(p)=1:\iff p=000...$,
\item $\lim:\In\IN^\IN\to\IN^\IN,\langle p_0,p_1,p_2,...\rangle\mapsto\lim_{n\to\infty}p_n$,
\item $\lim_X:\In X^\IN\to X,(x_n)_{n\in\IN}\mapsto\lim_{n\to\infty}x_n$,
\item $\J:\IN^\IN\to\IN^\IN,p\mapsto p'$,
\item $\EC:\IN^\IN\to2^\IN,p\mapsto\range(p-1)$,
\item $\C_X:\In\IN^\IN\mto X,p\mapsto X\setminus\range(p-1)$, $\dom(\C_X)=\{p\in\IN^\IN:|X\setminus\range(p-1)|\geq1\}$,
\item $\ACC_X:\In\IN^\IN\mto X,p\mapsto X\setminus\range(p-1)$ with\\
        $\dom(\ACC_X)=\{p\in\IN^\IN:|\range(p-1)|\leq1\}$,
\item $\AoUC_X:\In\IN^\IN\mto X,p\mapsto X\setminus\range(p-1)$ with\\
         $\dom(\AoUC_X)=\{p\in\IN^\IN:|X\setminus\range(p-1)|=1\mbox{ or }\range(p-1)=\emptyset\}$,
\item $\DNC_X:\IN^\IN\mto X^\IN,p\mapsto\{q\in X^\IN:(\forall i\in\IN)\;\varphi^p_i(i)\not=q(i)\}$,
\item $\PA:\IN^\IN\mto\IN^\IN,p\mapsto\{q\in\IN^\IN:p\ll q\}$,
\item $\WKL:\In\Tr\mto 2^\IN,T\mapsto[T]$ with $\dom(\WKL)=\{T:T$ infinite$\}$,
\item $\NON:\IN^\IN\mto\IN^\IN,p\mapsto\{q\in\IN^\IN:q\nleqT p\}$,
\item $\DIS:\IN^\IN\mto\IN^\IN,p\mapsto\{q\in\IN^\IN:\U(p)\not=q\}$,
\item $\NRNG:\IN^\IN\mto2^\IN,p\mapsto\{A\in2^\IN:A\not=\range(p-1)\}$.
\end{enumerate}
\end{defi}

Here $\LPO$ is also known as {\em limited principle of omniscience} and it is nothing but the characteristic
function of the zero sequence. By $\lim$ we just denote the ordinary limit map with respect to the
Baire space topology, where for convenience, the input sequence is encoded by
$\langle p_0,p_1,p_2,... \rangle\langle n,k\rangle:=p_n(k)$ where $\langle n,k\rangle:=\frac{1}{2}(n+k)(n+k+1)+k$
is the usual Cantor pairing function for $n,k\in\IN$.
By $p'$ we denote the {\em Turing jump} of $p\in\IN^\IN$.
We identify $n\in\IN$ with the set $n=\{0,...,n-1\}$.
The problem $\EC$ (this name was introduced in \cite[Exercise~8.2]{Wei00}) was originally
studied under the name $\C$ \cite{Ste89,Myl92,Bra99,Bra05,Myl06,BG11}. Intuitively,
$\EC$ translates enumerations into characteristic functions.
The problem $\C_\IN$ is known as choice on the natural numbers and was introduced and studied in \cite{BG11a,BBP12}.
The problems $\ACC_X$ are also known under the name $\LLPO_X$ and have been studied
in \cite{Wei92c,HK14a,BHK17a}. The acronym $\ACC$ stands for {\em all-or-co-unique choice} and $\LLPO:=\C_2=\ACC_2=\AoUC_2$
is known as {\em lesser limited principle of omniscience}. We recall that $p-1$ was defined above and
$|A|$ denotes the cardinality of the set $A$. 
The acronym $\AoUC$ stands for {\em all-or-unique choice}. This problem was studied mostly for the unit interval $X=[0,1]$ \cite{Pau11,BGH15a,KP16b}
and not  for spaces $X\In\IN$ that we are interested in here. 
The acronym $\DNC$ stands for {\em diagonally non-computable} and
by $\varphi^p$ we denote a standard
G\"odel numbering of the partial computable functions $\varphi^p_i:\In\IN\to\IN$ relative to some
oracle $p\in\IN^\IN$. 
The acronym $\PA$ stands for {\em Peano arithmetic} and 
by $p\ll q$ we express the fact that $q$ is of PA--degree relative to $p$,
which means that $q$  computes a path through every infinite binary tree that is computable relative to $p$.
The relation $\ll$ was introduced by Simpson~\cite{Sim77}.
By $\Tr$ we denote the set of binary trees $T\In\{0,1\}^*$ and $[T]$ denotes the set of infinite paths of such a tree
and $\WKL$ stands for {\em Weak K\H{o}nig's Lemma}.
By $\leqT$ we denote Turing reducibility and by $\U:\In\IN^\IN\to\IN^\IN$ we denote
some universal computable function. Such a function can be defined, for instance, by
$\U\langle\langle i,r\rangle,p\rangle:=\varphi^{\langle r,p\rangle}_i$ whenever $\varphi^{\langle r,p\rangle}_i$ is total
(and undefined otherwise).\footnote{The universal function $\U$ has been defined
differently in \cite{Bra20}, but our definition is equivalent, as the function $\U$ defined here also
satisfies a utm- and an smn-theorem~\cite{Wei00}. In particular, the discontinuity problem $\DIS$ defined
with one version of $\U$ is strongly Weihrauch equivalent to the one defined with the other version of $\U$.} 
Here $\langle i,r\rangle:=ir$ for $i\in\IN$ and $r\in\IN^\IN$.
The problems $\NON$ and $\DIS$ are called the {\em non-computability problem} and the {\em discontinuity problem},
respectively. The problem $\NRNG$ is called {\em range non-equality problem} and is introduced here.

Nobrega and Pauly used Wadge games to characterize certain lower cones in the Weihrauch lattice~\cite{NP19}.
In \cite{Bra20} we have characterized the upper cone of the discontinuity problem by Wadge games on problems.
The characterization goes as follows~\cite[Theorem~17, Corollary~28]{Bra20}.

\begin{thm}[Wadge games and the discontinuity problem]
\label{thm:DIS}
Let $f:\In X\mto Y$ be a problem. Then
$\DIS\leqW f\iff$ Player I has a computable winning strategy in the Wadge game $f$.
\end{thm}

An analogous result holds for $\leq_\mathrm{W}^*$ and (not necessarily computable) winning
strategies~\cite[Theorem~27]{Bra20}. We are going to use Theorem~\ref{thm:DIS} in the
proof of Proposition~\ref{prop:Sigma-ACC} that establishes the pentagon of $\ACC_\IN$ shown in Figure~\ref{fig:ACCN-pentagon}.

Several facts are known about the parallelization of the problems summarized in Definition~\ref{def:problems}.
These results were proved in~\cite[Lemma~6.3, Theorem~8.2]{BG11}, \cite[Lemma~8.9]{BBP12} and the result $\widehat{\ACC_X}\equivSW\DNC_X$
has first been proved by Higuchi and Kihara~\cite[Proposition~81]{HK14a} and independently in~\cite[Theorem~5.2]{BHK17a}.
See also the survey \cite{BGP21}.

\begin{fact}[Parallelization of problems]
\label{fact:parallelization-problem}
$\widehat{\LPO}\equivSW\widehat{\C_\IN}\equivSW\widehat{\lim_X}\equivSW\lim\equivSW\J\equivSW\EC$, \linebreak
$\widehat{\LLPO}\equivSW\widehat{\C_n}\equivSW\WKL$, and $\widehat{\ACC_X}\equivSW\DNC_X$
for $X=\IN$ or $X\geq2$ and $n\geq2$.
\end{fact}

Until recently the problems $\ACC_\IN$ and $\NON$ were the two weakest unsolvable (mutually incomparable)
natural problems in the Weihrauch lattice, and, in fact, they are the two weakest problems discussed in~\cite{BHK17a}.
Hence, it is a somewhat surprising coincidence that these problems appear together in the diagram
in Figure~\ref{fig:ACCN-pentagon}. In fact, they are related through $\NON\equivSW\Pi\Sigma(\ACC_\IN)$.
The discontinuity problem $\DIS$ was introduced in \cite{Bra20} and it was proved that (under the axiom of determinacy)
$\DIS$ is actually the weakest discontinuous problem with respect to the topological version of Weihrauch
reducibility $\leq_\mathrm{W}^*$. Part of the above relation between $\ACC_\IN$ and $\NON$ is that we
are going to prove that $\DIS$ parallelizes to $\NON$ (see Theorem~\ref{thm:DIS-NON}).

\begin{thm}[Non-computability is parallelized discontinuity]
$\NON\equivSW\widehat{\DIS}$.
\end{thm}

This result supports the slogan that ``non-computability is parallelized discontinuity''. 
In Section~\ref{sec:discontinuity-problem} with study a number of further pentagons with the discontinuity
problem at the bottom and we show that the discontinuity problem can be obtain by stashing of
several different problems.

\begin{thm}[Discontinuity as stashing]
$\DIS\equivSW\uwidehat{\LPO}\equivSW\uwidehat{\LLPO}\equivSW\uwidehat{\ACC_\IN}\equivSW\uwidehat{\AoUC_n}$ for $n\geq2$.
\end{thm}

In Section~\ref{sec:discontinuity-problem} we prove that this also leads to a number of further interesting 
characterizations of the discontinuity problem. 
Most notably we obtain a fully set-theoretic characterization (i.e., a characterization that does not refer to
any computability theoretic notion) with the following result.

\begin{prop}[Range-non-equality problem]
$\DIS\equivSW\NRNG$.
\end{prop}

In Section~\ref{sec:Turing-cone} we introduce and study another interior operator in the Weihrauch lattice
that replaces a problem by its upper Turing cone version.

\begin{defi}[Upper Turing cone version]
Let $f:\In X\mto Y$ be a problem.
We define the {\em upper Turing cone version} $f^\DD:\In X\mto\DD$ by
$\dom(f^\DD):=\dom(f)$ and
\[f^\DD(x):=\{\deg_\mathrm{T}(q)\in\DD:(\exists y\leqT q)\;y\in f(x)\}.\]
\end{defi}

Here $\DD$ denotes the set of {\em Turing degrees}
$\deg_\mathrm{T}(p)=\{q\in\IN^\IN:q\equivT p\}$, represented by
$\delta_\DD:\IN^\IN\to\DD,p\mapsto\deg_\mathrm{T}(p)$.
If $(Y,\delta_Y)$ is a represented space, then we define
\[y\leqT q:\iff(\exists p\in\delta_Y^{-1}\{y\})\;p\leqT q\iff(\exists\mbox{ computable }F:\In\IN^\IN\to Y)\,F(q)=y.\]
Hence, $y\leqT q$ means that $y$ has a name that can be computed from $q$ and if $Y=\IN^\IN$
then this is the usual version of Turing reducibility. 
This version of Turing reducibility has also been called {\em representation reducibility}~\cite{Mil04}.

Besides the fact that $f\mapsto f^\DD$ is an interior
operator in the Weihrauch lattice, our main result in this direction shows that on parallelizable problems
the two interior operator $f\mapsto\Sigma f$ and $f\mapsto f^\DD$ coincide (up to equivalence).

\begin{prop}[Closure under upper Turing cones]
$\Sigma\widehat{f}\equivSW\widehat{f}\;^\DD$ for all problems $f$.
\end{prop}

As a corollary of this result we obtain the following surprising consequence (see Corollary~\ref{cor:closure-properties}).

\begin{cor}[Closure under upper Turing cones]
If $f$ is a problem that is parallelizable and stashable, then it is also closed under upper Turing cones,
i.e., $f\equivW f^\DD$.
\end{cor}

This means that all such problems are essentially of computability theoretic nature.
This includes all problems that occur on the right-hand side of stashing-parallelization pentagons. 
The remarkable situation here is that parallelization and stashing are in some sense purely
set-theoretic operations (with no mention of any computability theoretic property or notion) and 
yet a combination of both generates computability-theoretic problems (starting from any problem $f$ whatsoever). 

In Section~\ref{sec:HYP} we study the {\em non-majorization problem} $\NMAJ$ as an asymmetric version of
the discontinuity problem and we investigate its pentagon. This is of interest as $\NMAJ$ parallelizes to the
hyperimmunity problem $\HYP$. 
In Section~\ref{sec:RET} we investigate the {\em retraction problem} $\RET_X$ that characterizes
the complexity of multi-valued retractions $R:\overline{X}\mto X$ and turns out to be equivalent
to $\AoUC_X$ for $X\geq 2$. We also study the corresponding pentagons.
Finally, in Section~\ref{sec:lim} we identify a phase transition related to the {\em limit avoidance problem} $\NLIM_\IN$.
This is the weakest problem known to us that does not stash away to the discontinuity problem.

\section{Stashing as Interior Operator}
\label{sec:stashing}

The main purpose of this section is to show that stashing $f\mapsto\Sigma f$ is an interior
operator for various versions of the Weihrauch lattice (see Proposition~\ref{prop:stashing}).
For this result it is essential that the completion $\overline{Y}$ of $Y$ is used on the output side
of $\Sigma f:\In X^\IN\mto\overline{Y}^\IN$. As a technical preparation we need the following lemma.
In \cite[Corollary~2.7]{BG21a} we proved that there is a computable retraction $r:\overline{\overline{Y}}\to\overline{Y}$,
which is a computable map such that $r|_{\overline{Y}}=\id_{\overline{Y}}$.
Here we will have to use a similar property of the product space $\overline{Y}^\IN$.

\begin{lem}[Retractions for product spaces]
\label{lem:product-retraceable}
For every represented space $Y$ there is a computable retraction $r:\overline{\overline{Y}^\IN}\to\overline{Y}^\IN$,
i.e., a computable $r$ with $r|_{\overline{Y}^\IN}=\id_{\overline{Y}^\IN}$.
\end{lem}
\begin{proof}
The space $Z:=\overline{Y}$ is represented by a total representation $\delta_Z=\delta_{\overline{Y}}$ and 
by \cite[Corollary~2.7]{BG21a} there is a computable retraction $R:\overline{Z}\to Z$.
We assume that $\overline{Z^\IN}$ is represented by the completion $\delta_{\overline{Z^\IN}}$ of the usual product representation $\delta_{Z^\IN}$.
We claim that there is a computable map $T:\overline{Z^\IN}\to\overline{Z}^\IN$ with $T|_{Z^\IN}=\id_{Z^\IN}$.
We assume that $\overline{Z^\IN}=Z^\IN\cup\{\bot_\IN\}$ and $\overline{Z}=Z\cup\{\bot\}$.
The names of $\bot_\IN$ with respect to $\delta_{\overline{Z^\IN}}$ are exactly those names that contain only finitely many digits different from $0$
(since $\delta_Z$ and hence $\delta_{Z^\IN}$ are total). Now $T$ can be realized in a computable way by interpreting the non-zero content
of a given name $p$ as a name $\langle p_0,p_1,p_2,...\rangle$ of a point in $\overline{Z}^\IN$ with respect to $\delta_{\overline{Z}^\IN}$.
As long as no non-zero content is available in $p$, the names $p_i$ are filled up by zeros. Altogether this shows that $T$ is computable and it acts
as the identity with the exception that $T(\bot_\IN)=(\bot,\bot,\bot,...)$. Now $r:\overline{Z^\IN}\to Z^\IN$ with $r:=\widehat{R}\circ T$
is the desired computable retraction for $Z^\IN$. 
\end{proof}

Now we are prepared to prove that stashing is an interior operator on the (strong) Weihrauch lattice.
Properties (1) and (2) in the following result are only made possible by the usage of the completion $\overline{Y}$, whereas
the existence of a retraction according to Lemma~\ref{lem:product-retraceable} guarantees that the
completion is not an obstacle for property (3). For a problem ${f:\In X\mto Y}$ the {\em completion} is defined
by
\[\overline{f}:\overline{X}\mto\overline{Y},x\mapsto\left\{\begin{array}{ll}
f(x) & \mbox{if $x\in\dom(f)$}\\
\bot & \mbox{otherwise}
\end{array}\right.
\]
This completion was introduced and studied in \cite{BG20,BG21a} and we will use it in part (2) of the following proof.

\begin{prop}[Stashing]
\label{prop:stashing}
The stashing operation $f\mapsto\Sigma f$ is an interior operator for $\leqSW$, $\leqW$, $\leq_{\rm W}^*$ and $\leq_{\rm sW}^*$.
That is, for all problems $f,g$:
\begin{enumerate}
\item $\Sigma f\leqSW f$,
\item $f\leqSW g\TO \Sigma f\leqSW \Sigma g$,
\item $\Sigma f\leqSW\Sigma\Sigma f$.
\end{enumerate}
Analogous statements hold for $\leqW$, $\leq_{\rm W}^*$ and $\leq_{\rm sW}^*$.
\end{prop}
\begin{proof}
If we prove (1) and (3) for $\leqSW$, then this implies the corresponding statements for $\leq_\mathrm{sW}^*$, $\leqW$ and $\leq_\mathrm{W}^*$.
Only in the case of (2) we explicitly need to consider the different types of reducibilities.\\
(1) We consider the computable projection $K:X^\IN\to X,(x_n)\mapsto x_0$
and the computable function $H:Y\to\overline{Y}^\IN,y\mapsto(y,\bot,\bot,...)$.
Here $H$ is computable, since the embedding ${\iota:Y\to\overline{Y}}$ is computable according to \cite[Corollary~3.10]{BG20}
and the element $\bot\in\overline{Y}$ is computable too.
These two functions $K,H$ witness the reduction $\Sigma f\leqSW f$, i.e., 
$HfK(x_n)\in\Sigma f(x_n)$ for all $(x_n)\in\dom(\Sigma f)$.
We note that the usage of the completion $\overline{Y}$ guarantees the existence of computable default outputs $\bot\in\overline{Y}$.\\
(2) We consider problems $f:\In X\mto Y$ and $g:\In W\mto Z$.
If $f\leqSW g$, then there are computable $K:\In X\mto W$ and $H:\In Z\mto Y$
such that $\emptyset\not=HgK(x)\In f(x)$ for all $x\in\dom(f)$~\cite[Proposition~3.2]{BGP21}.
Then the completion $\overline{H}:\overline{Z}\mto\overline{Y}$ is computable by \cite[Proposition~4.9]{BG20} and so are the 
parallelizations $\widehat{\overline{H}}:\overline{Z}^\IN\mto\overline{Y}^\IN$ and $\widehat{K}:\In X^\IN\mto W^\IN$. We obtain
\begin{eqnarray*}
\emptyset&\not=&\widehat{\overline{H}}\circ\Sigma g\circ\widehat{K}((x_n)_n)\\
&=& \widehat{\overline{H}}\{(z_n)_n\in\overline{Z}^\IN:(\exists n)\;z_n\in gK(x_n)\}\\
&=&\{(y_n)\in \overline{Y}^\IN:(\exists n)\;y_n\in HgK((x_n)_n)\}\In\Sigma f((x_n)_n)
\end{eqnarray*}
for all $(x_n)_n\in\dom(\Sigma f)$. 
We note that the totality of the completion $\overline{H}$ guarantees that also those components of $(z_n)_n\in\Sigma g\circ \widehat{K}((x_n)_n)\In\overline{Z}^\IN$ are processed, which are not in the domain of $H$.
Altogether, this proves $\Sigma f\leqSW\Sigma g$.\\
Let now $f\leqW g$ hold with computable $K:\In X\mto \IN^\IN\times W$ and $H:\In \IN^\IN\times Z\mto Y$ according to~\cite[Lemma~2.5]{BP18},
i.e., $\emptyset\not=H\circ(\id_{\IN^\IN}\times g)\circ K(x)\In f(x)$ for all $x\in\dom(f)$.
Then, again, the completion $\overline{H}:\overline{\IN^\IN\times Z}\mto\overline{Y}$ is computable.
By \cite[Proposition~3.8, Corollary~3.10]{BG20} there is a computable $\iota:\IN^\IN\times\overline{Z}\to\overline{\IN^\IN\times Z}$ 
with $\iota(p,z)=(p,z)$ for all $(p,z)\in \IN^\IN\times Z$.
Then $\widehat{(\overline{H}\circ\iota)}:(\IN^\IN\times\overline{Z})^\IN\mto\overline{Y}^\IN$ can also be considered
as a problem of type $H':(\IN^\IN)^\IN\times\overline{Z}^\IN\mto\overline{Y}^\IN$ and $\widehat{K}$ can be seen as a problem of type $\widehat{K}:\In X^\IN\mto(\IN^\IN)^\IN\times W^\IN$.
As above we obtain 
\[\emptyset\not=H'\circ(\id_{(\IN^\IN)^\IN}\times\Sigma g)\circ\widehat{K}((x_n)_n)\In\Sigma f((x_n)_n),\]
for all $(x_n)_n\in\dom(\Sigma f)$, i.e., $\Sigma f\leqW\Sigma g$.
The statements for $\leq_\mathrm{sW}^*$ and $\leq_\mathrm{W}^*$ can be proved analogously with continuous $K,H$.
We note that by \cite[Proposition~4.9]{BG20} $\overline{H}$ is continuous, if $H$ is so.\\
(3) We consider $\Sigma\Sigma f:\In (X^\IN)^\IN\mto\overline{\overline{Y}^\IN}^\IN$.
By Lemma~\ref{lem:product-retraceable} there is a computable retraction $r:\overline{\overline{Y}^\IN}\to\overline{Y}^\IN$.
For every represented space $X$ the map 
\[s_X:(X^\IN)^\IN\to X^\IN,((x_{n,k})_k)_n\mapsto(x_{n,k})_{\langle n,k\rangle}\] 
that interleaves a double sequence in a single sequence, is a computable isomorphism, i.e., it is bijective and computable and its inverse is computable too.
We now consider the computable functions $K:=s_X^{-1}$ and $H:=s_Y\circ\widehat{r}:\overline{\overline{Y}^\IN}^\IN\to\overline{Y}^\IN$.
Then we obtain
\begin{eqnarray*}
\emptyset &\not=& H\circ\Sigma\Sigma f\circ K((x_{n,k})_{\langle n,k\rangle})\\
&=& s_Y\circ \widehat{r}\{(z_n)_n\in\overline{\overline{Y}^\IN}^\IN:(\exists n)\;z_n\in\Sigma f((x_{n,k})_k)\}\\
&=& s_Y\{((y_{n,k})_k)_n\in(\overline{Y}^\IN)^\IN:(\exists n)(\exists k)\;y_{n,k}\in f(x_{n,k})\}\\
&=& \{(y_{n,k})_{\langle n,k\rangle}\in\overline{Y}^\IN:(\exists\langle n,k\rangle)\;y_{n,k}\in f(x_{n,k})\}\\
&=&\Sigma f((x_{n,k})_{\langle n,k\rangle}).
\end{eqnarray*}
This proves $\Sigma f\leqSW\Sigma\Sigma f$.
\end{proof}

We can conclude from property (2) in Proposition~\ref{prop:stashing} that stashing is, in particular, invariant under (strong) Weihrauch reducibility
and can hence be seen as an operation on (strong) Weihrauch degrees.

\begin{cor}
Stashing can be extended to an operation on (strong) Weihrauch degrees.
\end{cor}

\section{The Stashing-Parallelization Monoid}
\label{sec:monoid}

We adopt the convention that we denote the parallelization of a problem $f$ by $\widehat{f}$ and the
stashing by $\uwidehat{f}$ when we deal with single applications of these operators. However, 
for iterated applications it is useful to use the notation $\Pi f$ and $\Sigma f$ instead.

The closure and interior operators $\Pi$ and $\Sigma$ generate a monoid $\{\Pi,\Sigma\}^*$ under composition
and we want to study the action of this monoid on the Weihrauch lattice. 
To this end, it is worth spelling out the problems $\Sigma\Pi f$ and $\Pi\Sigma f$ explicitly:
\begin{enumerate}
\item $\Sigma\Pi f:\In X^{\IN\times\IN}\mto\overline{Y^\IN}^\IN,(x_{n,k})\mapsto\{(y_{n,k}):(\exists n)(\forall k)\;y_{n,k}\in f(x_{n,k})\}$,
\item $\Pi\Sigma f:\In X^{\IN\times\IN}\mto\overline{Y}^{\IN\times\IN},(x_{n,k})\mapsto\{(y_{n,k}):(\forall n)(\exists k)\;y_{n,k}\in f(x_{n,k})\}$.
\end{enumerate}

We can see that stashing corresponds to a usage of an existential quantifier whereas parallelization corresponds
to a usage of a universal quantifier in a certain sense. Hence $\Sigma\Pi$ and $\Pi\Sigma$ correspond to applications
of these quantifiers in different order.

There is a subtle technical point here: we have to deal with the spaces $\overline{Y}^\IN$ and $\overline{Y^\IN}$,
which are not computably isomorphic. We recall that $\overline{Y}^\IN=(Y\cup\{\bot\})^\IN$, whereas
$\overline{Y^\IN}=Y^\IN\cup\{\bot_\IN\}$. Hence, formally there is no subset relation between these two
sets. In order to make the latter a subset of the former, we can choose $\bot_\IN:=(\bot,\bot,...)$, as implicitly done in the proof of Lemma~\ref{lem:product-retraceable}.
In this sense the double sequence notation $(y_{n,k})$ in the description of $\Sigma\Pi f$ should be understood.

Besides the retraction $r$ from Lemma~\ref{lem:product-retraceable} we also need the maps $s,t$ that exist according to the following lemma.
Intuitively speaking, $s$ maps every sequence that contains a $\bot$ to $\bot_\IN$ and $t$ maps
$\bot_\IN$ to some sequence that contains a $\bot$ 
(which one it is, might depend on the given name of $\bot_\IN$).

\begin{lem}[Completion of product spaces]
\label{lem:completion-product}
For every represented space $Y$ there are computable 
$s:\overline{Y}^\IN\to\overline{Y^\IN}$ and $t:\overline{Y^\IN}\mto\overline{Y}^\IN$
such that $s|_{Y^\IN}=t|_{Y^\IN}=\id_{Y^\IN}$.
\end{lem}
\begin{proof}
A suitable computable map $s:\overline{Y}^\IN\to\overline{Y^\IN}$ with $s|_{Y^\IN}=\id|_{Y^\IN}$ 
is realized by a computable $F:\IN^\IN\to\IN^\IN$ with the property that $F\langle p_0,p_1,p_2,...\rangle-1=\langle q_0,q_1,q_2,...\rangle$ with $q_i=p_i-1$ for all $p_i$ with $p_i-1\in\IN^\IN$.
This can be achieved by copying the non-zero content of $p_i$ subtracted by $1$ into the $q_i$, where the resulting sequence $\langle q_0,q_1,q_2,...\rangle$ is filled up by zeros, whenever necessary (i.e., whenever
no non-zero content is available for some $p_i$ then the entire output is filled up only with zeros as long as no non-zero content appears). That is, if one of the $p_i$ is a name of $\bot\in\overline{Y}$
(either because it has only finitely many digits different from zero or because $p_i-1\not\in\dom(\delta_Y)$), then $F\langle p_0,p_1,...\rangle$
is a name of $\bot_\IN\in\overline{Y^\IN}$.
In this way, $F$ realizes the identity on $Y^\IN$.

For the second part of the statement, we note that $\overline{Y}^\IN$ has a precomplete and total representation by \cite[Proposition~3.8]{BG20}
and hence we can extend the parallelization of the computable embedding $Y\into\overline{Y}$ to a computable
problem $t:\overline{Y^\IN}\mto\overline{Y}^\IN$ with $t|_{Y^\IN}=\id_{Y^\IN}$ by~\cite[Proposition~2.6]{BG21a}.
\end{proof}

Our core observation on the action of the monoid $\{\Pi,\Sigma\}^*$ on the (strong) Weihrauch 
lattice is captured by the following result.

\begin{prop}[Action of the stashing-parallelization monoid]
\label{prop:monoid}
For every problem $f$ we obtain:
\begin{enumerate}
\item $\Pi\Sigma f\leqSW\Sigma\Pi f$,
\item $\Pi\Sigma\Pi f\equivSW\Sigma\Pi f$,
\item $\Sigma\Pi\Sigma f\equivSW\Pi\Sigma f$.
\end{enumerate}
\end{prop}
\begin{proof}
(1) 
Given an instance $(x_{n,k})$ of $\Pi\Sigma f$, we just swap $n$-- with $k$--positions in $(x_{n,k})$, 
then we apply $\Sigma\Pi f$ to the result, then we use the parallelization of the problem $t:\overline{Y^\IN}\mto\overline{Y}^{\IN}$ 
from Lemma~\ref{lem:completion-product} in order to convert the output of $\Sigma\Pi f$ from $(\overline{Y^\IN})^\IN$ into a double
sequence in $(\overline{Y}^\IN)^\IN$ and then we swap the $n$-- and $k$--positions again to
obtain a result in $\Pi\Sigma f(x_{n,k})$, which is correct because
\[(\exists k)(\forall n)\;y_{n,k}\in f(x_{n,k})\TO(\forall n)(\exists k)\;y_{n,k}\in f(x_{n,k}).\]
(2) Since $\Pi$ is a closure operator we have $\Sigma\Pi f\leqSW\Pi\Sigma\Pi f$ and $\Pi\Pi f\leqSW\Pi f$, and
by (1) and monotonicity of $\Sigma$ we obtain $\Pi\Sigma\Pi f\leqSW\Sigma\Pi\Pi f\leqSW\Sigma\Pi f$.\\
(3) Since $\Sigma$ is an interior operator we have $\Sigma\Pi\Sigma f\leqSW\Pi\Sigma f$ and $\Sigma f\leq\Sigma\Sigma f$, and 
by (1) and monotonicity of $\Pi$ we obtain
$\Pi\Sigma f\leqSW\Pi\Sigma\Sigma f\leqSW\Sigma\Pi\Sigma f$.
\end{proof}

That is the action of the stashing-parallelization monoid $\{\Pi,\Sigma\}^*$ on a problem $f$ in 
the (strong) Weihrauch lattice leads to at most five distinct degrees $\{f,\Pi f,\Sigma f,\Sigma\Pi f,\Pi\Sigma f\}$,
which are arranged in a pentagon, see the diagram in Figure~\ref{fig:Pi-Sigma-pentagon}.
Every line in the diagram indicates a $\leqSW$--reduction in the upwards direction.

Of course, if a problem $f$ is parallelizable and stashable (such as any problem of the form $f=\Sigma\Pi g$ is),
then the pentagon reduces to a single degree. Other smaller sizes than five can be realized too and,
as we will see, also the maximal size of five can be realized. In any case, the {\em stashing-parallelization pentagon}
can be seen as the {\em trace} of $f$ under $\{\Sigma,\Pi\}^*$ 
that reveals some information about the underlying problem $f$.
It follows from Proposition~\ref{prop:monoid} and the fact that $\Pi$ and $\Sigma$ are closure and interior operators, respectively,
that for a problem $f$ with a full pentagon of size five, $f$ is incomparable with the opposite problems $\Sigma\Pi f$ and $\Pi\Sigma f$.

\begin{cor}[Full pentagons]
\label{cor:full-pentagon}
For every problem $f$ we obtain:
\begin{enumerate}
\item $f\leqSW\Sigma\Pi f\iff \Pi f\equivSW\Sigma\Pi f$,
\item $\Sigma\Pi f\leqSW f\iff \Sigma f\equivSW\Sigma\Pi f$,
\item $f\leqSW\Pi\Sigma f\iff \Pi f\equivSW\Pi\Sigma f$,
\item $\Pi\Sigma f\leqSW f\iff \Sigma f\equivSW\Pi\Sigma f$.
\end{enumerate}
Analogous statements hold for the reductions $\leqW,\leq_\mathrm{sW}^*,\leq_\mathrm{W}^*$.
\end{cor}

\begin{figure}
\begin{tikzpicture}[scale=2]
\node[style={fill=\colorf},left] at (-1.03,0) {$f$};
\node[fill,circle,scale=0.3] (f) at (-1,0) {};
\node[style={fill=\colorsf},below] at (-0.3,-0.981) {$\Sigma f$};
\node[fill,circle,scale=0.3] (Sf) at (-0.3,-0.951) {};
\node[style={fill=\colorpf},above] at (-0.3,0.981) {$\Pi f$};
\node[fill,circle,scale=0.3]  (Pf) at (-0.3,0.951) {};
\node[style={fill=\colorspf},right] at  (0.83,0.588)  {$\Sigma\Pi f$};
\node[fill,circle,scale=0.3] (SPf) at (0.8,0.588) {};
\node[style={fill=\colorpsf},right] at (0.83,-0.588) {$\Pi\Sigma f$};
\node[fill,circle,scale=0.3]  (PSf) at (0.8,-0.588) {};
\draw[thick] (Pf) -- (f);
\draw[thick] (Pf) -- (SPf);
\draw[thick] (f) -- (Sf);
\draw[thick] (SPf) -- (PSf);
\draw[thick] (PSf) -- (Sf);
\node at (-0.75,0.585) {$\Pi$};
\node at (0.38,0.88) {$\Sigma$};
\node at (-0.75,-0.585) {$\Sigma$};
\node at (0.38,-0.88) {$\Pi$};
\end{tikzpicture}
\caption{Parallelization-stashing pentagon in the Weihrauch lattice.} 
\label{fig:Pi-Sigma-pentagon}
\end{figure}
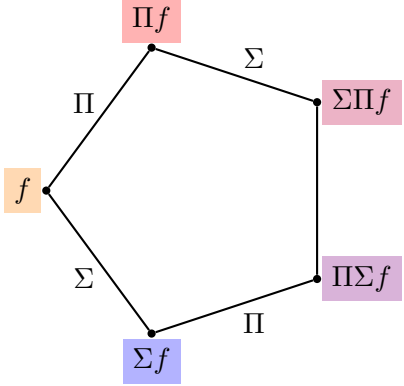

There are many interesting questions regarding the interaction of parallelization and stashing.
For instance, we will see later in Corollary~\ref{cor:injectivity} that the map $f\mapsto(\Sigma f,\Pi f)$
is not injective on Weihrauch degrees, i.e., problems are not characterized by their respective pentagons.
However, these pentagons still reveal some interesting information in many cases, as we will see.

\section{The Upper Turing Cone Operator}
\label{sec:Turing-cone}

The main purpose of this section is to prove that the upper Turing cone operator $f\mapsto f^\DD$ is an
interior operator on the (strong) Weihrauch lattice that coincides with the interior operator $f\mapsto\Sigma f$,
restricted to (strongly) parallelizable problems.

In the following it is useful to have a simplified version of $\Sigma\Pi f$, where we replace double sequences on the input side by ordinary sequences. For this purpose
we consider the injection $I_X:X\into X^\IN,x\mapsto(x,x,x,...)$ for every represented space $X$.

\begin{lem}
\label{lem:Sigma-Pi}
$\Sigma\Pi f\equivSW(\Sigma\Pi f)\circ I_{X^\IN}$ for every problem $f:\In X\mto Y$.
\end{lem}
\begin{proof}
Since $I_{X^\IN}$ is computable, it is clear that $(\Sigma\Pi f)\circ I_{X^\IN}\leqSW\Sigma\Pi f$ holds.
On the other hand, given an instance $((x_{n,k})_k)_n\in X^{\IN^2}$ of $\Sigma\Pi f$, we can compute the 
interleaved sequence $(x_{n,k})_{\langle n,k\rangle}\in X^\IN$ and then apply $I_{X^\IN}$ followed by $\Sigma\Pi f$. 
This yields a sequence $((y_{m,\langle n,k\rangle})_{\langle n,k\rangle})_m\in\overline{Y^\IN}^\IN$
that satisfies the property $(\exists m)(\forall\langle n,k\rangle)\;y_{m,\langle n,k\rangle}\in f(x_{n,k})$.
Using the parallelization of the problem $t:\overline{Y^\IN}\mto\overline{Y}^\IN$ from Lemma~\ref{lem:completion-product} we can convert
this into a sequence in $\overline{Y}^{\IN^2}$ and then extract a diagonal sequence
$((y_{n,\langle n,k\rangle})_{\langle n,k\rangle})_n\in\overline{Y}^{\IN^2}$ from it.
This sequence can be converted back to $\overline{Y^\IN}^\IN$ by the parallelization of the function $s:\overline{Y}^\IN\to\overline{Y^\IN}$ 
from Lemma~\ref{lem:completion-product} and we claim that the result
is a solution of $\Sigma\Pi f(((x_{n,k})_k)_n)$. This is because
\[(\exists m)(\forall\langle n,k\rangle)\;y_{m,\langle n,k\rangle}\in
  f(x_{n,k})\TO(\exists n)(\forall k)\;y_{n,\langle n,k\rangle}\in f(x_{n,k}). \qedhere\]
\end{proof}

The exact relation between the upper Turing cone operator and stashing is captured in the following result,
again with the help of the injection $I_X$.

\begin{prop}[Upper Turing cone operator]
\label{prop:Turing-cone-version}
$f^\DD\equivSW\Sigma f\circ I_X$ holds for every problem $f:\In X\mto Y$. In particular, $f^\DD\leqSW\Sigma f$.
\end{prop}
\begin{proof}
We consider the  represented spaces $(X,\delta_X)$ and $(Y,\delta_Y)$.
We claim that every realizer $F:\In\IN^\IN\to\IN^\IN$ of $\Sigma f\circ I_X$ is also
a realizer of $f^\DD$, which proves $f^\DD\leqSW\Sigma f\circ I_X$.
To this end, let $F$ be such a realizer and $x:=\delta_X(p)\in\dom(f^\DD)=\dom(f)$.
Let $(y_n):=\delta_{\overline{Y}^\IN}F(p)\in\Sigma f(x,x,...)$.
Let $\langle q_0,q_1,...\rangle:=F(p)$ and let $G:\In\IN^\IN\to\IN^\IN$ be a 
computable realizer of the partial inverse $\iota^{-1}:\In \overline{Y}\to Y$ of the embedding $\iota:Y\to\overline{Y}$,
which is computable according to \cite[Corollary~3.10]{BG20}.
Then there is some $n\in\IN$ with
$y_n=\delta_YG(q_n)=\delta_{\overline{Y}}(q_n)\in f\delta_X(p)=f(x)$ and $y_n\leqT G(q_n)\leqT F(p)$.
This proves the claim.\\
For the reverse reduction $\Sigma f\circ I_X\leqSW f^\DD$ we first note that
$\delta_{\overline{Y}}$ is a precomplete representation by \cite[Proposition~3.8]{BG20}. 
Hence there is a total computable $r:\IN\to\IN$ such that for all $q\in\IN^\IN$, $n\in\IN$
the function $\varphi^q_{r(n)}$ is always total and if $\varphi^q_n$ is total and $\varphi^q_n\in\dom(\delta_Y)$, then $\delta_{\overline{Y}}(\varphi^q_{r(n)})=\delta_{Y}(\varphi^q_n)$.
Intuitively, the programme with code $r(n)$ works as the programme $n$,
but it adds $1$ to all output results and fills up the output with dummy symbols $0$ in appropriate positions
as long as no other 
better information becomes available. 
Now we consider the computable function $H:\IN^\IN\to\IN^\IN$
with $H(q):=\langle\varphi^q_{r(0)},\varphi^q_{r(1)},...\rangle$.
We claim that $HF$ is a realizer of $\Sigma f\circ I_X$ for every realizer $F$ of $f^\DD$.
To this end, let $F$ be such a realizer and $x:=\delta_X(p)\in\dom(\Sigma f\circ I_X)=\dom(f)$.
Then there exists some $y\leqT F(p)$ with $y\in f(x)$. Hence, there is some $n\in\IN$
such that $y=\delta_Y(\varphi^{F(p)}_n)=\delta_{\overline{Y}}(\varphi^{F(p)}_{r(n)})$.
This implies that $\delta_{\overline{Y}^\IN}HF(p)\in\Sigma f\circ I_X(x)$, which completes the proof.
\end{proof}

We note that the main idea of the proof, namely to compute on all G\"odel numbers in parallel, 
can only be realized because stashing uses a completion $\overline{Y}$ of the space $Y$ on the output side.

By a combination of Propositions~\ref{prop:Turing-cone-version} and \ref{prop:monoid} with Lemma~\ref{lem:Sigma-Pi}
we obtain the following corollary.

\begin{cor}[Upper Turing cone operator]
\label{cor:Turing-cone-version}
$(\Pi f)^\DD\equivSW\Sigma\Pi f$ and $(\Pi\Sigma f)^\DD\equivSW\Pi\Sigma f$ for every problem $f$.
\end{cor}

This means that both problems on the right-hand side of the diagram in Figure~\ref{fig:Pi-Sigma-pentagon}
can be seen as upper Turing cone versions and hence as computability-theoretic problems.
We emphasize that Turing cones appear here out of a purely topological context without any computability theory being involved.
This is because Corollary~\ref{cor:Turing-cone-version}
is also correct when the computability-theoretic Weihrauch reducibility is replaced by its topological counterpart.

We mention in passing that the upper Turing cone operator is an interior operator on the Weihrauch lattice.

\begin{prop}[Upper Turing cone operator as interior operator]
\label{prop:Turing-cone-interior}
The operation $f\mapsto f^\DD$ is an interior operator on the (strong) Weihrauch lattice.
That is, for all problems $f,g$ we have:
\begin{enumerate}
\item $f^\DD\leqSW f$,
\item $f\leqSW g\TO f^\DD\leqSW g^\DD$,
\item $f^\DD\leqSW f^{\DD\DD}$.
\end{enumerate}
Analogous statements hold for $\leqW$, $\leq_{\rm W}^*$ and $\leq_{\rm sW}^*$.
\end{prop}
\begin{proof}
We consider problems $f:\In X\mto Y$ and $g:\In W\mto Z$.\\
(1) This follows from $f^\DD\leqSW\Sigma f\leqSW f$, which holds by
Propositions~\ref{prop:Turing-cone-version} and \ref{prop:stashing}.\\
(2) Let $f\leqW g$ hold via computable $H,K:\In\IN^\IN\to\IN^\IN$, i.e.,
$H\langle\id,GK\rangle$ is a realizer of $f$ whenever $G$ is a realizer of $g$.
We claim that $f^\DD\leqW g^\DD$ holds via $\id,K$.
Let $p$ be a name of some input $x\in\dom(f)$. Then $K(p)$ is a
name of a point $w\in\dom(g)$ and any name $q$ of a point in $g(w)$
yields a name $H\langle p,q\rangle$ of a point in $f(x)$, since there is a realizer
$G$ of $g$ with $GK(p)=q$. Let now $G$ be a realizer of $g^\DD$. Then
there is a name $q$ of a point in $g(w)$ such that $q\leqT GK(p)$.
Hence $H\langle p,q\rangle\leqT\langle p,GK(p)\rangle$.
This shows that $\langle\id,GK\rangle$ is a realizer of $f^\DD$ whenever
$G$ is a realizer of $g^\DD$ and hence $f^\DD\leqW g^\DD$.
The statement for $\leqSW$ can be proved analogously.
In the topological cases we have to work with continuous $H,K$.
Then $H$ is computable relative to some $r\in\IN^\IN$ and we obtain
as above $H\langle p,q\rangle\leqT\langle r,p,GK(p)\rangle$.
Hence $f^\DD\leq_\mathrm{W}^*g^\DD$ holds via continuous $H',K$, where
$H'\langle p,q\rangle:=\langle r,p,q\rangle$.\\
(3) We have even $f^{\DD\DD}=f^\DD$ by transitivity of Turing reducibility.
\end{proof}

By Proposition~\ref{prop:monoid} stashing extends to an interior operator
on parallelizable Weihrauch degrees and by Corollary~\ref{cor:Turing-cone-version} 
the upper Turing cone operator coincides on those degrees with stashing.

\begin{cor}
\label{cor:Turing-cone-version-interior}
$f\mapsto\Sigma f$ and $f\mapsto f^\DD$ are identical interior operators restricted to (strongly) parallelizable (strong) Weihrauch degrees.
\end{cor}

We note that $f\mapsto\Sigma f$ and $f\mapsto f^\DD$ are not identical on arbitrary Weihrauch degrees.
The problem $f^\DD$ is always computable when $f$ has only computable solutions. For instance, 
$\LPO^\DD$ is computable, while this is not the case for $\Sigma(\LPO)$ (see Proposition~\ref{prop:Sigma-LPO}).

We can also formulate this result such that problems which are simultaneously stashable and parallelizable
are automatically closed under applying the upper Turing cone operator.

\begin{cor}
\label{cor:closure-properties}
For every problem $f$ the following conditions are equivalent to each other:
\begin{enumerate}
\item $\Sigma f\equivSW f$ and $\Pi f\equivSW f$,
\item $f^\DD\equivSW f$ and $\Pi f\equivSW f$.
\end{enumerate}
An analogous property holds with $\equivW$ instead of $\equivSW$.
\end{cor}

It follows from Corollary~\ref{cor:Turing-cone-version} that problems $g$ that are simultaneously parallelizable
and stashable can only occur in certain regions of the Weihrauch lattice. For one, every problem with the set
of Turing degrees as target set is {\em densely realized} by \cite[Corollary~4.9]{BHK17a}, which means
that a realizer of such a problem can produce outputs that start with arbitrary prefixes.
This in turn implies by \cite[Proposition~6.3]{BP18} that any problem with discrete output below it has to be computable.
We formulate this as a corollary.

\begin{cor}[Parallelizable and stashable problems]
\label{cor:parallelizable-stashable}
Let $f:\In X\mto\IN$ be a problem and let $g$ be a problem that is parallelizable and stashable.
If $f\leqW g$ holds, then $f$ is computable.
\end{cor}

One of the weakest problems with discrete output that is discontinuous is $\ACC_\IN$,
the all-or-co-unique choice problem on $\IN$. 
This problem was studied in~\cite{BHK17a} and an equivalent problem was investigated earlier
under the name $\LLPO_X$~\cite{Wei92c,HK14} (and under the name $\LLPO_\infty$~\cite[Definition~16]{Myl06} in the case of $\ACC_\IN$).
Intuitively speaking, $\ACC_\IN$ is the problem that given a list of natural numbers which is either empty
or contains exactly one number, one has to produce a number which is not in the list.
For $f=\ACC_\IN$ we can also phrase Corollary~\ref{cor:parallelizable-stashable} as follows.

\begin{cor}[The cone of all-or-co-unique choice]
\label{cor:cone-ACC}
If $\ACC_\IN\leqW g$ holds for some problem $g$, then $g$ cannot be simultaneously parallelizable and stashable.
\end{cor}

Hence, in a certain sense, problems that are parallelizable and stashable at the same time are rare, even rarer than this result suggests.
Namely, $\ACC_\IN$ is not the weakest discontinuous problem with discrete output.
Mylatz has proved that there are also discontinuous problems of type $f:\In\IN^\IN\mto\IN$ with $f\lW\ACC_\IN$~\cite[Satz~14]{Myl06}.

\section{The Discontinuity Problem in Pentagons}
\label{sec:discontinuity-problem}

In this section we investigate the discontinuity problem $\DIS$ by studying a number of stashing-parallelization pentagons
in which it appears as the bottom problem. 
Along the line we will formulate some problems that are equivalent to $\DIS$.
In the following we use the notation $\uwidehat{f}=\Sigma f$ for the stashing of specific problems $f$.
We start with defining a number of problems related to $\ACC_\IN$.

\begin{defi}[Problems related to all-or-co-unique choice]
\label{def:Sigma-ACC}
We consider the following problems:
\begin{enumerate}
\item $A:\IN^\IN\mto\IN,\langle \langle i,n\rangle,p\rangle\mapsto\{k\in\IN:\varphi^p_i(n)\not=k\}$,
\item $B:\IN^\IN\mto\overline{\IN}^\IN,\langle i,p\rangle\mapsto\{q\in\overline{\IN}^\IN:(\exists n)\;\varphi_i^p(n)\not=q(n)\in\IN\}$,
\item $C:\IN^\IN\mto\IN^\IN,\langle i,p\rangle\mapsto\{q\in\IN^\IN:(\exists n)\;\varphi_i^p(n)\not=q(n)\}$.
\end{enumerate}
\end{defi}

As a first result we prove that the discontinuity problem is the stashing of $\ACC_\IN$.

\begin{prop}[All-or-co-unique choice]
\label{prop:Sigma-ACC}
We obtain
$\uwidehat{\ACC_\IN}\equivSW\uwidehat{A}\equivSW B\equivSW C\equivSW\DIS$ and $\ACC_\IN\equivSW A$.
\end{prop}
\begin{proof}
It is straightforward to see that $\ACC_\IN\equivSW A$, which implies $\uwidehat{\ACC_\IN}\equivSW\uwidehat{A}$ by Proposition~\ref{prop:stashing}.

We prove $\uwidehat{A}\equivSW B$. 
We note that $\uwidehat{A}$ is of type $\uwidehat{A}:(\IN^\IN)^\IN\mto\overline{\IN}^\IN$.
In order to show $\uwidehat{A}\leqSW B$, we consider instances
\[p=(\langle\langle i_0,n_0\rangle,p_0\rangle,\langle\langle i_1,n_1\rangle,p_1\rangle,\langle\langle i_2,n_2\rangle,p_2\rangle,...)\in(\IN^\IN)^\IN\]
of $\uwidehat{A}$.
There is a $j\in\IN$ such that $\varphi_{j}^{\langle p\rangle}(k)=\varphi^{p_k}_{i_k}(n_k)$ for all $k\in\IN$ and all $p$ of the above form.
Hence, a solution to $B\langle j,\langle p\rangle\rangle$ is a solution to $\uwidehat{A}(p)$.
Thus $\uwidehat{A}\leqSW B$. 
For the inverse reduction we consider the computable function
$K:\IN^\IN\to(\IN^\IN)^\IN$ with $K\langle i,p\rangle:=(\langle\langle i,0\rangle,p\rangle,\langle\langle i,1\rangle,p\rangle,\langle\langle i,2\rangle,p\rangle,...)$.
This function reduces $B$ to $\uwidehat{A}$, i.e., we obtain $\uwidehat{A}\equivSW B$.

It is easy to see that $B\leqSW C$ holds, since the parallelization $\widehat{\iota}:\IN^\IN\to\overline{\IN}^\IN$ of the embedding $\iota:\IN\to\overline{\IN}$ is computable, where $\iota$ is computable according to \cite[Corollary~3.10]{BG20}.

We now prove $C\equivSW\DIS$.
The reduction $C\leqSW\DIS$ follows with help of the computable function $K:\IN^\IN\to\IN^\IN$ with $K\langle i,\langle r,p\rangle\rangle:=\langle\langle i,r\rangle,p\rangle$. The computable
inverse of $K$ yields $\DIS\leqSW C$. We note that these reductions are obvious when the corresponding $\varphi_i^{\langle r,p\rangle}$
is total, but otherwise any $q\in\IN^\IN$ is allowed as a solution in both cases.

Finally, we prove $\DIS\leqSW B$. By Theorem~\ref{thm:DIS} it suffices to show that player I has a computable winning strategy
in the Wadge game $B$. Therefore we consider a G\"odel number $i\in\IN$ such that $\varphi_i^p(n)=k$ if and only if $\langle n,k\rangle+1$ is the first number of the form $\langle n,m\rangle+1$ listed in $p$.
If there is no number of this form, then $\varphi_i^p(n)$ is undefined. In other words, the program $i$ upon input $n$ and oracle $p$ searches for the first number of the form $\langle n,k\rangle+1$ listed in $p$ 
and outputs $k$ if such a number is found. Now player I in the Wadge game $B$ starts playing $\langle i,p_0\rangle$ with $p_0=000...$. This corresponds to the nowhere defined function $\varphi_i^{p_0}$
and hence player II must play a name of some $q_0\in\overline{\IN}^\IN$ with some $n\in\IN$ such that $q_0(n)\in\IN$; otherwise player II looses. However, when the first candidate $n_0\in\IN$ with $q_0(n_0)\in\IN$ appears, 
then player I modifies the $p_0$ in its play to a $p_1$ by appending $\langle n_0,q_0(n_0)\rangle$ in the current position to it. 
Hence $\varphi_i^{p_1}(n_0)=q_0(n_0)$.
This forces player II to modify its play to a name of some $q_1$ with another $n_1\in\IN$ with $q_1(n_1)\in\IN$ and $n_0\not=n_1$; otherwise player II looses.
Now player I modifies $p_1$ by appending $\langle n_1,q_1(n_1)\rangle$ to it. 
This strategy can be continued inductively and describes a computable winning strategy for player I. 
\end{proof}

We can conclude from Proposition~\ref{prop:Sigma-ACC} that the discontinuity problem is stashable.

\begin{cor}
\label{cor:summability-DIS}
$\DIS$ is strongly stashable.
\end{cor}

The next result supports the slogan that ``non-computability is the parallelization of (effective) discontinuity''.

\begin{thm}[Discontinuity and non-computability]
\label{thm:DIS-NON}
$\widehat{\DIS}\equivSW\NON$.
\end{thm}
\begin{proof}
By Proposition~\ref{prop:Sigma-ACC} it suffices to show $\widehat{C}\equivSW\NON$. For this purpose it is helpful to reformulate $\NON$ as follows:
\[\NON:\IN^\IN\mto\IN^\IN,p\mapsto\{q\in\IN^\IN:(\forall i)(\exists n)\;\varphi^p_i(n)\not=q(n)\}.\]
We note that this formulation of $\NON$ is equivalent to the usual one, as for non-total functions $\varphi^p_i$ there exists always an $n\in\IN\setminus\dom(\varphi_i^p)$ that satisfies $\varphi^p_i(n)\not=q(n)$,
since $q$ is total.
With $\NON$ written in this form it is clear that $C\leqSW\NON$.
Moreover, $\NON$ is strongly parallelizable, since given $p:=\langle p_0,p_1,...\rangle$ it is clear that $q\not\leq_{\rm T}p$ implies $q\not\leq_{\rm T}p_i$ for every $i\in\IN$.
Hence, $\widehat{C}\leqSW\NON$. 

For the inverse reduction $\NON\leqSW\widehat{C}$, we assume that we have given some $p\in\IN^\IN$.
Then we can evaluate $\widehat{C}$ on the instance $(p_i)_{i\in\IN}$ with $p_i:=\langle i,p\rangle$ in order to
get some output $(q_i)_{i\in\IN}\in\widehat{C}(p_i)_{i\in\IN}$ with the property that $(\forall i)(\exists n)\;\varphi^p_i(n)\not=q_i(n)$.
We claim that $q:=\langle q_0,q_1,q_2,...\rangle\not\leq_{\rm T}p$. If we assume the contrary, then
there is some total computable $r:\IN\to\IN$ such that $\varphi^p_{r(i)}(n)=q_i(n)$ for all $i,n\in\IN$.
Hence, by the relativized version of Kleene's fixed point theorem~\cite[Theorem~2.2.1]{Soa16} there is some $i\in\IN$ with $\varphi^p_i(n)=\varphi^p_{r(i)}(n)=q_i(n)$ for all $n\in\IN$,
which contradicts the assumption that $q_i\in C\langle i,p\rangle$. Altogether, this proves $\widehat{\DIS}\equivSW\widehat{C}\equivSW\NON$.
\end{proof}

We note that $\DNC_\IN^\DD\lW\DNC_\IN$ follows  since
$\ACC_\IN\leqW\DNC_\IN$, but $\ACC_\IN\nleqW\DNC_\IN^\DD$ by Corollary~\ref{cor:cone-ACC}.
Together with Fact~\ref{fact:parallelization-problem}, we have established the pentagon of $\ACC_\IN$ given in Figure~\ref{fig:ACCN-pentagon}.
Perhaps the pentagon in Figure~\ref{fig:ACCN-pentagon} is the most natural pentagon in which the discontinuity problem $\DIS$ appears,
but it is by far not the only one. 
It was observed by Jockusch~\cite[Theorem~6]{Joc89} and Weihrauch~\cite[Theorem~4.3]{Wei92c} that
the problems $\DNC_n$ and $\ACC_n$, respectively, form strictly decreasing chains, i.e., we have the following fact (see also \cite[Corollary~82]{HK14a}, \cite[Corollary~3.8]{BHK17a}).

\begin{fact}[Jockusch 1989, Weihrauch 1992]
\label{fact:DNC-ACC}
For all $n\geq2$ we have:
\begin{enumerate}
\item $\DNC_\IN\lSW\DNC_{n+1}\lSW\DNC_n$,
\item $\ACC_\IN\lSW\ACC_{n+1}\lSW\ACC_n$.
\end{enumerate}
\end{fact}

\begin{figure}
\begin{tikzpicture}[scale=2]
\node[style={fill=\colorf},left] at (-1.03,0) {$\ACC_n$};
\node[fill,circle,scale=0.3] (f) at (-1,0) {};
\node[style={fill=\colorsf},below] at (-0.3,-0.981) {$\DIS$};
\node[fill,circle,scale=0.3] (Sf) at (-0.3,-0.951) {};
\node[style={fill=\colorpf},above] at (-0.3,0.981) {$\DNC_n$};
\node[fill,circle,scale=0.3]  (Pf) at (-0.3,0.951) {};
\node[style={fill=\colorspf},right] at  (0.83,0.588)  {$\PA$};
\node[fill,circle,scale=0.3] (SPf) at (0.8,0.588) {};
\node[style={fill=\colorpsf},right] at (0.83,-0.588) {$\NON$};
\node[fill,circle,scale=0.3]  (PSf) at (0.8,-0.588) {};
\draw[thick] (Pf) -- (f);
\draw[thick] (Pf) -- (SPf);
\draw[thick] (f) -- (Sf);
\draw[thick] (SPf) -- (PSf);
\draw[thick] (PSf) -- (Sf);
\node at (-0.75,0.585) {$\Pi$};
\node at (0.38,0.88) {$\Sigma$};
\node at (-0.75,-0.585) {$\Sigma$};
\node at (0.38,-0.88) {$\Pi$};
\end{tikzpicture}
\caption{$\ACC_n$ pentagon in the Weihrauch lattice for $n\geq2$.} 
\label{fig:ACCn-pentagon}
\end{figure}
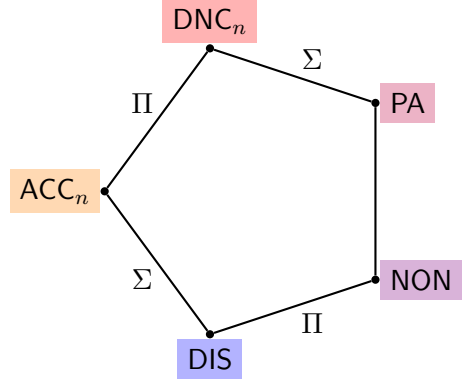

On the other hand, it turns out that the stashing of the problem $\ACC_n$ is strongly equivalent to $\DIS$ for all $n\geq2$.
In order to express this result, it is useful to consider a universal function of type $\U_n:\In\IN^\IN\to\{0,...,n-1\}^\IN$.
Such a function can be defined by truncating $\U$ accordingly:
\[\U_n\langle\langle i,r\rangle,p\rangle:=\min(n-1,\varphi^{\langle r,p\rangle}_i)=\min(n-1,\U\langle\langle i,r\rangle,p\rangle)\]
whenever $\varphi^{\langle r,p\rangle}_i$ is total  (where the minimum is understood pointwise).
Using this definition we can also modify the problem $\DIS$ accordingly and in this way we obtain
\[\DIS_n:\IN^\IN\mto\{0,...,n-1\}^\IN,p\mapsto\{q\in\{0,...,n-1\}^\IN:\U_n(p)\not=q\}\] 
for all $n\geq2$.
In these terms we obtain the following result.

\begin{prop}[All-or-co-unique choice]
\label{prop:Sigma-ACCn}
$\uwidehat{\ACC_n}\equivSW \DIS_n\equivSW\DIS$ for all $n\geq2$.
\end{prop}
\begin{proof}
By Fact~\ref{fact:DNC-ACC} it suffices to consider the case $n=2$. The remaining cases follow
by Proposition~\ref{prop:Sigma-ACC} since stashing is an interior operator by Proposition~\ref{prop:stashing}.
The reduction $\DIS\equivSW\uwidehat{\ACC_\IN}\leqSW\uwidehat{\ACC_2}$ also follows.
The reduction $\uwidehat{\ACC_2}\leqSW\DIS_2$ can be proved almost literally following 
the lines of the proof of Proposition~\ref{prop:Sigma-ACC} with some obvious modifications.
For instance, one needs to replace
all terms $\varphi^p_i(n)$ in the definitions of $A,B$ and $C$ by $\min(1,\varphi^p_i(n))$;
one has to replace the output types of $A,B$ and $C$ by  $\{0,1\}$, $\overline{\{0,1\}}^\IN$ and $\{0,1\}^\IN$, respectively; 
and one has to work with the embedding $\iota:\{0,1\}\to\overline{\{0,1\}}$.
It remains to prove the reduction $\DIS_2\leqSW\DIS$. For this direction we use the computable embedding
$\iota:\IN^\IN\to2^\IN,p\mapsto0^{p(0)}10^{p(1)}10^{p(2)}1...$.
There is a total computable $s:\IN\to\IN$ such that $\iota(\varphi^t_{s(i)})=\min(1,\varphi^t_{i})$ for all $i\in\IN$ and $t\in\IN^\IN$
such that $\varphi^t_i$ is total and $\min(1,\varphi^t_{i})$ contains infinitely many ones.
Now, given an instance $\langle\langle i,r\rangle,p\rangle$ of $\DIS_2$ we compute the instance $\langle\langle s(i),r\rangle, p\rangle$ of $\DIS$.
If $q\in\IN^\IN$ satisfies $q\not=\U\langle\langle s(i),r\rangle, p\rangle$, then $\iota(q)\not=\U_2\langle\langle i,r\rangle, p\rangle$ follows.
This is clear if $\varphi^{\langle r,p\rangle}_i$ is total and $\min(1,\varphi^{\langle r,p\rangle}_{i})$ contains infinitely many ones. But otherwise every $q\in\IN^\IN$
satisfies the conclusion. Altogether, this completes the proof.
\end{proof}

The upper Turing cone version of $\DNC_n$ for $n\geq 2$ is just the problem $\PA$ of finding a PA degree relative to the input.
By a result of Jockusch and Friedberg~\cite[Theorem~5]{Joc89} the Turing degrees of $q\gg p$ are exactly the degrees that compute
a diagonally non-computable function $f:\IN\to\{0,...,n-1\}$ relative to $p$ for every $n\geq 2$ (see also \cite[Proposition~6.1]{BHK17a}).
Hence we obtain \cite[Corollary~6.4]{BHK17a} the following fact.

\begin{fact}
\label{fact:PA}
$\PA\equivSW\DNC_n^\DD$ for every $n\geq2$.
\end{fact}

We note that $\PA\lW\DNC_n$ follows  since
$\ACC_n\leqW\DNC_n$, but $\ACC_n\nleqW\PA$ by Corollary~\ref{cor:parallelizable-stashable}.
Altogether, we have thus established the pentagon of $\ACC_n$ for $n\geq2$ given in Figure~\ref{fig:ACCn-pentagon}.
An important special case of this diagram is the case for $n=2$. 
Since $\LLPO\equivSW\ACC_2$, Fact~\ref{fact:parallelization-problem} yields the stashing-parallelization pentagon
of $\LLPO$ given in Figure~\ref{fig:LLPO-pentagon}.

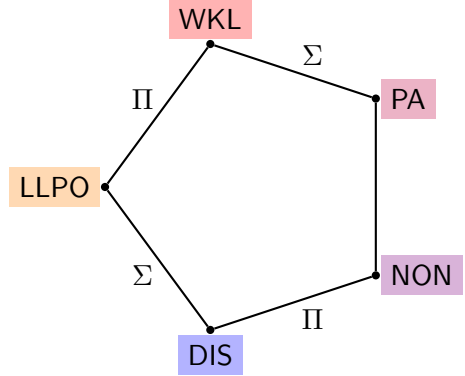
\begin{figure}
\begin{tikzpicture}[scale=2]
\node[style={fill=\colorf},left] at (-1.03,0) {$\LLPO$};
\node[fill,circle,scale=0.3] (f) at (-1,0) {};
\node[style={fill=\colorsf},below] at (-0.3,-0.981) {$\DIS$};
\node[fill,circle,scale=0.3] (Sf) at (-0.3,-0.951) {};
\node[style={fill=\colorpf},above] at (-0.3,0.981) {$\WKL$};
\node[fill,circle,scale=0.3]  (Pf) at (-0.3,0.951) {};
\node[style={fill=\colorspf},right] at  (0.83,0.588)  {$\PA$};
\node[fill,circle,scale=0.3] (SPf) at (0.8,0.588) {};
\node[style={fill=\colorpsf},right] at (0.83,-0.588) {$\NON$};
\node[fill,circle,scale=0.3]  (PSf) at (0.8,-0.588) {};
\draw[thick] (Pf) -- (f);
\draw[thick] (Pf) -- (SPf);
\draw[thick] (f) -- (Sf);
\draw[thick] (SPf) -- (PSf);
\draw[thick] (PSf) -- (Sf);
\node at (-0.75,0.585) {$\Pi$};
\node at (0.38,0.88) {$\Sigma$};
\node at (-0.75,-0.585) {$\Sigma$};
\node at (0.38,-0.88) {$\Pi$};
\end{tikzpicture}
\caption{$\LLPO$ pentagon in the Weihrauch lattice.} 
\label{fig:LLPO-pentagon}
\end{figure}

We have a number of basic discrete problems ordered in the following way~\cite[Fact~3.4]{BHK17a}, \cite[Theorem~3.10]{BG11a}, \cite[Section~13]{BGM12}.

\begin{fact}
\label{fact:ACC-LLPO-LPO-CN}
$\ACC_\IN\leqW\LLPO\leqW\LPO\leqW\lim_2\leqW\C_\IN$.
\end{fact}

Now the question appears how far up in this chain of discrete problems
we can go such that we still obtain the discontinuity problem $\DIS$ as stashing
of the corresponding discrete problem?
We will see in Proposition~\ref{prop:NLIM-WKL} that a phase transition in this respect
happens between $\LPO$ and $\lim_2$.

We now study the pentagon of $\LPO$.
We use the notation $\W^p_i:=\dom(\varphi^p_i)$ and by $\chi_A:\IN\to\{0,1\}$
we denote the {\em characteristic function} of $A\In\IN$ with $A=\chi_A^{-1}\{1\}$.
We first define some problems related to $\LPO$.

\begin{defi}[Problems related to $\LPO$]
\label{def:Sigma-LPO}
We consider:
\begin{enumerate}
\item $L:\IN^\IN\to\{0,1\},\langle \langle i,n\rangle,p\rangle\mapsto1-\chi_{\W^p_i}(n)=\left\{\begin{array}{cl}0 & \mbox{if $n\in\dom(\varphi^p_i)$}\\ 1 & \mbox{otherwise}\end{array}\right.$,
\item $D:\IN^\IN\mto\overline{\{0,1\}}^\IN,\langle i,p\rangle\mapsto\{q\in\overline{\{0,1\}}^\IN:(\exists n)\;\chi_{\W^p_i}(n)\not=q(n)\in\{0,1\}\}$,
\item $E:\IN^\IN\mto\{0,1\}^\IN,\langle i,p\rangle\mapsto\{q\in\{0,1\}^\IN:(\exists n)\;\chi_{\W^p_i}(n)\not=q(n)\}$.
\end{enumerate}
\end{defi}

Now we can prove the following result.

\begin{prop}[Stashing of $\LPO$]
\label{prop:Sigma-LPO}
$\LPO\equivSW L$ and $\uwidehat{\LPO}\equivSW\uwidehat{L}\equivSW D\equivSW E\equivSW\DIS$.
\end{prop}
\begin{proof}
We proceed as in the proof of Proposition~\ref{prop:Sigma-ACC}. It is easy to see that $\LPO\equivSW L$, which implies $\uwidehat{\LPO}\equivSW\uwidehat{L}$.
The same reductions that prove $\uwidehat{A}\equivSW B$ in Proposition~\ref{prop:Sigma-ACC} also show $\uwidehat{L}\equivSW D$.
The reduction $D\leqSW E$ follows using the computable embedding ${\iota:\{0,1\}\to\overline{\{0,1\}}}$.
By Proposition~\ref{prop:Sigma-ACC} and Fact~\ref{fact:ACC-LLPO-LPO-CN} and since stashing is an interior operator by Proposition~\ref{prop:stashing},
we obtain $\DIS\equivSW\uwidehat{\ACC_\IN}\leqSW\uwidehat{\LPO}$.

It only remains to show $E\leqSW\DIS$. By the proof of Proposition~\ref{prop:Sigma-ACCn} it suffices to show $E\leqSW C$, where
\[C:\IN^\IN\mto\{0,1\}^\IN,\langle i,p\rangle\mapsto\{q\in\{0,1\}^\IN:(\exists n)\;\min(1,\varphi^p_i(n))\not=q(n)\}\]
is the modification of the function from Proposition~\ref{prop:Sigma-ACC} that was used in the proof of Proposition~\ref{prop:Sigma-ACCn}
in order to show $\DIS\equivSW\DIS_2\equivSW C$. For the reduction $E\leqSW C$ we proceed as follows.
Given an instance $\langle i,p\rangle$ of $E$, we try to find out for each $n\in\IN$, which of the two consecutive values $2n, 2n+1$ appears in $\W^p_i$ first, if any.
More precisely, there is a computable function $r:\IN\to\IN$ such that
\[\varphi^p_{r(i)}(n)=\left\{\begin{array}{ll}
0 & \mbox{if $2n\in\W^p_i$ is found first}\\
1 & \mbox{if $2n+1\in\W^p_i$ is found first}\\
\uparrow & \mbox{if $\{2n,2n+1\}\cap\W^p_i=\emptyset$}
\end{array}\right.\]
holds for all $i,n\in\IN$ and $p\in\IN^\IN$.
We use the computable function $K:\IN^\IN\to\IN^\IN$ with $K\langle i,p\rangle=\langle r(i),p\rangle$ to translate instances 
of $E$ into instances of $C$ and the computable function $H:\{0,1\}^\IN\to\{0,1\}^\IN$ with $H(q)(2n):=1-q(n)$ and $H(q)(2n+1):=q(n)$
in order to translate solutions of $C$ into solutions of $E$. That this reduction is correct can be seen as follows.
Given an instance $\langle i,p\rangle$ of $E$ and $q\in C\langle r(i),p\rangle=CK\langle i,p\rangle$, there is some $n\in\IN$ with $\min(1,\varphi^p_{r(i)}(n))\not=q(n)$.
We are now in exactly one of the following three cases:
\begin{enumerate}
\item $\varphi^p_{r(i)}(n)=0\TO (q(n)=1$ and $2n\in\W^p_i)\TO H(q)(2n)=0\not=\chi_{\W^p_i}(2n)$,
\item $\varphi^p_{r(i)}(n)=1\TO (q(n)=0$ and $2n+1\in\W^p_i)$ \\
         $\TO H(q)(2n+1)=0\not=\chi_{\W^p_i}(2n+1)$,
\item $\varphi^p_{r(i)}(n)=\,\uparrow\;\TO(q(n)\in\{0,1\}$ and $\{2n,2n+1\}\cap\W^p_i=\emptyset)$\\
        $\TO(H(q)(2n)=1\not=\chi_{\W^p_i}(2n)$ or $H(q)(2n+1)=1\not=\chi_{\W^p_i}(2n+1))$.
\end{enumerate}
In any case we obtain $H(q)\in E\langle i,p\rangle$.
Altogether we obtain $E\leqSW C\equivSW\DIS$.
\end{proof}

By Fact~\ref{fact:parallelization-problem} we have $\widehat{\LPO}\equivSW\lim\equivSW\J$.
It is clear that $\J^\DD$ is strongly
Weihrauch equivalent to  
\[\J^\DD:\DD\mto\DD,a\mapsto\{b\in\DD:a'\leqT b\}.\]
We note that $\J^\DD\lW\lim$ follows  since
$\LPO\leqW\lim$, but $\LPO\nleqW\J^\DD$ by Corollary~\ref{cor:parallelizable-stashable}.
Altogether, we have established the pentagon of $\LPO$ given in Figure~\ref{fig:LPO-pentagon}.

\begin{figure}
\begin{tikzpicture}[scale=2]
\node[style={fill=\colorf},left] at (-1.03,0) {$\LPO$};
\node[fill,circle,scale=0.3] (f) at (-1,0) {};
\node[style={fill=\colorsf},below] at (-0.3,-0.981) {$\DIS$};
\node[fill,circle,scale=0.3] (Sf) at (-0.3,-0.951) {};
\node[style={fill=\colorpf},above] at (-0.3,0.981) {$\lim$};
\node[fill,circle,scale=0.3]  (Pf) at (-0.3,0.951) {};
\node[style={fill=\colorspf},right] at  (0.83,0.588)  {$\J^\DD$};
\node[fill,circle,scale=0.3] (SPf) at (0.8,0.588) {};
\node[style={fill=\colorpsf},right] at (0.83,-0.588) {$\NON$};
\node[fill,circle,scale=0.3]  (PSf) at (0.8,-0.588) {};
\draw[thick] (Pf) -- (f);
\draw[thick] (Pf) -- (SPf);
\draw[thick] (f) -- (Sf);
\draw[thick] (SPf) -- (PSf);
\draw[thick] (PSf) -- (Sf);
\node at (-0.75,0.585) {$\Pi$};
\node at (0.38,0.88) {$\Sigma$};
\node at (-0.75,-0.585) {$\Sigma$};
\node at (0.38,-0.88) {$\Pi$};
\end{tikzpicture}
\caption{$\LPO$ pentagon in the Weihrauch lattice.} 
\label{fig:LPO-pentagon}
\end{figure}
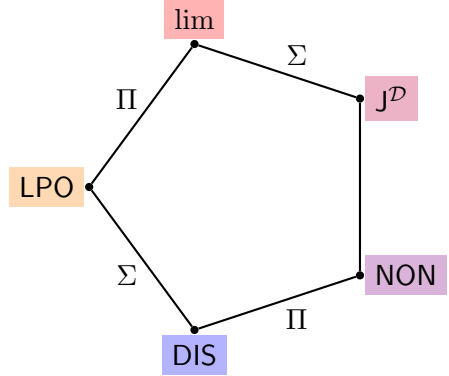

Proposition~\ref{prop:Sigma-LPO} also leads to another characterization of the discontinuity problem
in terms of ranges. This characterization is unique among all the characterizations that we have provided
because it is purely set-theoretic (i.e., no G\"odel  numberings or other computability-theoretic concepts are used)
and because it only involves standard data types (i.e., no completions are mentioned).

\begin{defi}[Range non-equality problem]
We call 
\[\NRNG:\IN^\IN\mto2^\IN,p\mapsto\{A\in2^\IN:A\not=\range(p-1)\}\]
the {\em range non-equality problem}.
\end{defi}

We now obtain the following characterization.

\begin{cor}[Range non-equality problem]
$\DIS\equivSW\NRNG$.
\end{cor}
\begin{proof}
By Proposition~\ref{prop:Sigma-LPO} it suffices to show $\NRNG\equivSW E$.
For one, there is a $j\in\IN$ such that $\W^p_j=\range(p-1)$, which establishes the reduction $\NRNG\leqSW E$. 
On the other hand, there is a computable $f:\IN^\IN\to\IN^\IN$ such that $\range(f\langle i,p\rangle-1)=\W^p_i$ for all $p\in\IN^\IN$ and $i\in\IN$,
which shows $E\leqSW\NRNG$.
\end{proof}

The reader might have noticed that a lot of problems that occur in the lower parts of our pentagons 
can actually be seen as complementary problems of other well-known problems. We briefly make this more precise.

\begin{defi}[Complementary problem]
For every problem $f:\In X\mto Y$ we define the {\em complementary problem} $f^{\mathrm c}:\In X\mto Y$
by $\graph(f^\mathrm{c}):=\graph(f)^\mathrm{c}=(X\times Y)\setminus\graph(f)$.
\end{defi}

That is $\dom(f^\mathrm{c})=\{x\in X:f(x)\not=Y\}$ and $f(x):=Y\setminus f(x)$ for all $x\in\dom(f^\mathrm{c})$.
Using this concept we see that $\DIS=\U^\mathrm{c}$,  $\NRNG=\EC^\mathrm{c}$, and $\NON=(\geq_\mathrm{T})^\mathrm{c}$,
where $\geq_\mathrm{T}:\IN^\IN\mto\IN^\IN,p\mapsto\{q\in\IN^\IN:q\leqT p\}$.

Even though complementation yields a neat way of expressing these problems, $f\mapsto f^\mathrm{c}$ is not an operation on 
the Weihrauch lattice. For instance $\J^\mathrm{c}$ is obviously computable, whereas $\EC^\mathrm{c}\equivSW\DIS$ is not,
although $\J\equivSW\EC$ by Fact~\ref{fact:parallelization-problem}.

\section{Majorization and Hyperimmunity}
\label{sec:HYP}

In this section we study the stashing-parallelization pentagons of the non-majorization problem $\NMAJ$ that can be seen as an asymmetric version of the discontinuity problem.
The non-majorization problem $\NMAJ$ is introduced in the following definition 
and it is related to the well-known hyperimmunity problem.

\begin{defi}[Problems related to hyperimmunity]
\label{def:HYP}
We consider the following problems:
\begin{enumerate}
\item $\NGEQ:\IN^\IN\mto\IN,\langle\langle i,n\rangle,p\rangle\mapsto\{k\in\IN:\varphi_i^p(n)\not\geq k\}$,
\item $\NMAJ:\IN^\IN\mto\IN^\IN,\langle i,p\rangle\mapsto\{q\in\IN^\IN:(\exists n)\;\varphi_i^p(n)\not\geq q(n)\}$,
\item $\HYP:\IN^\IN\mto\IN^\IN,p\mapsto\{q\in\IN^\IN:(\forall r\leqT p)(\exists n)\;r(n)<q(n)\}$,
\item $\MEET:\IN^\IN\mto\IN^\IN,p\mapsto\{q\in\IN^\IN:(\forall r\leqT p)(\exists n)\;r(n)=q(n)\}$,
\item $1\mbox{-}\WGEN:\IN^\IN\mto\IN^\IN,p\mapsto\{q\in\IN^\IN:q$ is weakly $1$--generic relative to $p\}$. 
\end{enumerate}
\end{defi}

The {\em non-majorization problem} $\NMAJ$ has been defined here ad hoc, whereas the Weih\-rauch complexity of the {\em hyperimmunity problem} $\HYP$
and the {\em weak $1$--genericity problem} $1\mbox{-}\WGEN$ have already been studied in \cite{BHK17a,BHK18}.
The principle $\MEET$ was introduced in a reverse mathematics context in \cite{HRSZ17}.
We note that the existential quantifier ``$\exists n$'' in $\HYP$ and $\MEET$ could equivalently be replaced by ``$\exists^\infty n$''.
We recall that a point $p\in\IN^\IN$ is called
{\em weakly $1$--generic} relative to $q\in\IN^\IN$ if $p\in U$ for every dense open set $U\In\IN^\IN$ that is c.e.\ open relative to $q$.
A set $U\In\IN^\IN$ is {\em c.e.\ open relative to $q$} if $U=U_i^q:=\{p\in\IN^\IN:0\in\dom(\varphi_i^q)\}$ for some $i\in\IN$.
By a theorem of Kurtz the hyperimmune degrees coincide with the weakly $1$--generic degrees and this also holds uniformly
in the following sense~\cite[Corollary~9.5]{BHK17a}.

\begin{fact}[Uniform theorem of Kurtz]
\label{fact:HYP}
$\HYP\equivW 1\mbox{-}\WGEN$.
\end{fact}

On the first sight, the non-majorization problem $\NMAJ$ looks similar to the discontinuity problem $\DIS$
in the form of $C$, as defined in Definition~\ref{def:Sigma-ACC}. In fact, $\NMAJ$ can be seen as an asymmetric
version of $C$, since the inequality $\not=$ is simply replaced by $\not\geq$ (we note that $\not\geq$ is not the same as $<$ here, 
as $\varphi_i^p$ might be partial and $\varphi_i^p(n)\not\geq q(n)$ is supposed to mean that either $\varphi_i^p(n)$ does not exist or
$\varphi_i^p(n)$ exists and $\varphi_i^p(n)< q(n)$.)
Despite the similarity between $\NMAJ$ and $\DIS$, it turns out that $\NMAJ$ is neither equivalent to $\DIS$ nor stashable. 
Among all the problems that we have studied here, it is perhaps the one that comes closest to $\DIS$ without being equivalent
to it. The following result clarifies the relation of these problems to each other.

\begin{prop}[The non-majorization problem]
\label{prop:NMAJ}
We obtain
$\uwidehat{\NMAJ}\equivW\uwidehat{\NGEQ}\equivW\DIS$, $\widehat{\NMAJ}\equivSW\HYP$, $\DIS\lW\NMAJ\lW\HYP$, and $\NON\lW\HYP^\DD$.
In particular, $\NMAJ$ is not stashable.
\end{prop}
\begin{proof}
With $C$ from Definition~\ref{def:Sigma-ACC} we obtain $\DIS\equivSW C\leqSW\NMAJ\leqSW \NGEQ\leqW\LPO$. 
The latter reduction holds since $\LPO$ can be used to determine whether $\varphi_i^p(n)$ is defined
and if it is defined then one can use the original input to find a larger value; otherwise $0$ is a suitable output.
This implies $\DIS\equivW\uwidehat{\NGEQ}\equivW\uwidehat{\NMAJ}$ by Proposition~\ref{prop:Sigma-LPO} since
stashing is an interior operator by Proposition~\ref{prop:stashing}.
Moreover, we have
\[\HYP(p)=\{q\in\IN^\IN:(\forall r\leqT p)(\exists n)\;r(n)<q(n)\}=\{q\in\IN^\IN:(\forall i)(\exists n)\;\varphi_i^p(n)\not\geq q(n)\}.\]
Here clearly ``$\supseteq$'' holds regarding the second equality since the $\varphi_i^p$ include all the total $r\leqT p$ and the other inclusion ``$\In$''
holds as for $\varphi_i^p$ that are not total the condition $\varphi_i^p(n)\not\geq q(n)$ is satisfied by definition for all $n\not\in\dom(\varphi_i^p)$. 
We claim that $\widehat{\NMAJ}\equivSW\HYP$. 
For one, it is clear that $\NMAJ\leqSW\HYP$ holds. Moreover, $\HYP$ is strongly parallelizable,
as there is some computable $r:\IN\to\IN$ with $\varphi_{r\langle i,k\rangle}^{\langle p_0,p_1,p_2,...\rangle}(n)=\varphi_i^{p_k}(n)$ for all $i,n,k\in\IN$ and $p_0,p_1,...\in\IN^\IN$.
Together, this implies $\widehat{\NMAJ}\leqSW\HYP$.
On the other hand, there is a computable $s:\IN\to\IN$ such that $\varphi_{s(i)}^p(n)=\varphi^p_i\langle i,n\rangle$ for all $i,n\in\IN$ and $p\in\IN^\IN$.
Hence, the function $K:\IN^\IN\to\IN^\IN$ with $K(p):=\langle\langle s(0),p\rangle,\langle s(1),p\rangle,\langle s(2),p\rangle,...\rangle$ is computable and
with $q=\langle q_0,q_1,q_2,...\rangle\in\widehat{\NMAJ}\circ K(p)$ we obtain 
\[(\forall i)(\exists n)\;\varphi_i^p\langle i,n\rangle=\varphi^p_{s(i)}(n)\not\geq q_i(n)= q\langle i,n\rangle,\]
so in particular $q\in\HYP(p)$. This proves $\HYP\leqSW\widehat{\NMAJ}$.

Suppose $\DIS\equivW\NMAJ$, then $\NON\equivW\HYP$ would follow by Theorem~\ref{thm:DIS-NON}, since parallelization is a closure operator.
However, it is well-known that there are hyperimmune-free non-computable degrees~\cite[Section~2]{MM68}, i.e., there 
is non-computable $q$ which is not of hyperimmune degree.
The problem $\NON$ has a realizer that, on computable inputs, produces such non-computable $q$,
which is not of hyperimmune degree.
Since hyperimmune degrees are upwards closed by \cite[Theorem~1.1]{MM68}, 
we obtain that $\NON\lW\HYP^\DD\leqW\HYP$.
This implies $\DIS\lW\NMAJ$.
Finally, we clearly have $\NMAJ\lW\HYP$ as $\NMAJ$ has computable solutions on all instances,
while $\HYP$ does not. 
\end{proof}

We emphasize that we have only proved $\uwidehat{\NMAJ}\equivW\DIS$ with an ordinary
Weihrauch equivalence, unlike in all previous cases, where we have established a strong Weihrauch equivalence.
Hence we are left with the following open question.

\begin{qu}
Does $\uwidehat{\NMAJ}\equivSW\DIS$ hold?
\end{qu}

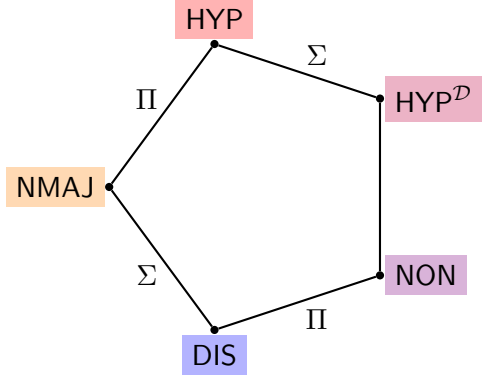
\begin{figure}
\begin{tikzpicture}[scale=2]
\node[style={fill=\colorf},left] at (-1.03,0) {$\NMAJ$};
\node[fill,circle,scale=0.3] (f) at (-1,0) {};
\node[style={fill=\colorsf},below] at (-0.3,-0.981) {$\DIS$};
\node[fill,circle,scale=0.3] (Sf) at (-0.3,-0.951) {};
\node[style={fill=\colorpf},above] at (-0.3,0.981) {$\HYP$};
\node[fill,circle,scale=0.3]  (Pf) at (-0.3,0.951) {};
\node[style={fill=\colorspf},right] at  (0.83,0.588)  {$\HYP^\DD$};
\node[fill,circle,scale=0.3] (SPf) at (0.8,0.588) {};
\node[style={fill=\colorpsf},right] at (0.83,-0.588) {$\NON$};
\node[fill,circle,scale=0.3]  (PSf) at (0.8,-0.588) {};
\draw[thick] (Pf) -- (f);
\draw[thick] (Pf) -- (SPf);
\draw[thick] (f) -- (Sf);
\draw[thick] (SPf) -- (PSf);
\draw[thick] (PSf) -- (Sf);
\node at (-0.75,0.585) {$\Pi$};
\node at (0.38,0.88) {$\Sigma$};
\node at (-0.75,-0.585) {$\Sigma$};
\node at (0.38,-0.88) {$\Pi$};
\end{tikzpicture}
\caption{Non-marjorization pentagon in the Weihrauch lattice.} 
\label{fig:NMAJ-pentagon}
\end{figure}

We note that $\HYP^\DD$ is exactly the problem of finding a hyperimmune degree relative to the input, i.e.,
it can equivalently be described as
\[\HYP^\DD:\DD\mto\DD,a\mapsto\{b\in\DD:b\mbox{ is of hyperimmune degree relative to $a$}\},\]
since hyperimmune degrees are upwards closed by \cite[Theorem~1.1]{MM68}.
This establishes the pentagon in Figure~\ref{fig:NMAJ-pentagon}, except that we did not yet
prove $\HYP^\DD\lW\HYP$. In the case of the earlier pentagons discussed here, we have used
Corollary~\ref{cor:parallelizable-stashable} for the corresponding separation. 
In the case of the hyperimmunity problem a more tailor-made argument is required,
since $\NMAJ$ has no natural number output and is densely realized itself.
We combine ideas from the proof of \cite[Proposition~6.3]{BP18} and the proof of Proposition~\ref{prop:Sigma-ACC}.

\begin{prop}
\label{prop:HYP}
$\HYP^\DD\lW\HYP$ (even restricted to computable instances).
\end{prop}
\begin{proof}
Here we consider $\HYP^\DD$ to be defined as 
\[\HYP^\DD:\IN^\IN\mto\IN^\IN,p\mapsto\{s\in\IN^\IN:(\exists q\leqT s)(\forall r\leqT p)(\exists n)\;r(n)<q(n)\}.\]
It suffices to prove $\NMAJ\nleqW\HYP^\DD$. Let us assume the contrary, i.e., let $H,K:\In\IN^\IN\to\IN^\IN$
be computable functions such that $H\langle\id,GK\rangle$ is a realizer of $\NMAJ$ whenever $G$ is a realizer of $\HYP^\DD$.
Then there is a computable monotone function $h:\IN^*\to\IN^*$ that approximates $H$ in the sense that
$H(q)=\sup_{w\prefix q}h(w)$ for all $q\in\dom(H)$. As in the proof of Proposition~\ref{prop:Sigma-ACC}
we consider a fixed G\"odel number $i\in\IN$ such that $\varphi_i^p(n)=k$ if and only if $\langle n,k\rangle+1$ is the first number of the form $\langle n,m\rangle+1$ listed in $p$.
Now we use $h$ and the finite extension method 
to construct an input $ip=\langle i,p\rangle\in\IN^\IN$ of $\NMAJ$ on which the above reduction fails.
We start with $p_0:=000...$, which yields the nowhere defined function $\varphi_i^{p_0}$. For this $p_0$ there is a lexicographically first $w_0\in\{2,3,4,...\}^*$
such that $|h\langle ip_0|_{|w_0|},d_0w_0\rangle|>0$ for both values $d_0\in\{0,1\}$. 
This is because there is some $s\in\{2,3,4,...\}^\IN$ with $d_0s\in\HYP^\DD K\langle i,p_0\rangle$ for both values $d_0\in\{0,1\}$. 
We choose 
\[a_0:=\max\{h\langle ip_0|_{|w_0|},d_0w_0\rangle(0):d_0\in\{0,1\}\}\mbox{ and }b_0:=\langle 0,a_0\rangle+1\]
and we continue
with $p_1:=0^{|w_0|}b_0000...$, which yields a function that satisfies $\varphi_i^{p_1}(0)=a_0$ and is undefined otherwise.
Again there is a lexicographically first $w_1\in\{2,3,4,...\}^*$  such that $|h\langle ip_1|_{|w_0|+|w_1|+1},d_0w_0d_1w_1\rangle|>1$ for all values $d_0,d_1\in\{0,1\}$
and now we choose
\[a_1:=\max\{h\langle ip_1|_{|w_0|+|w_1|+1},d_0w_0d_1w_1\rangle(1):d_0,d_1\in\{0,1\}\}\mbox{ and }b_1:=\langle 1,a_1\rangle+1.\]
The next input is $p_2:=0^{|w_0|}b_00^{|w_1|}b_1000...$, which represents a function $\varphi_i^{p_2}$ 
that satisfies $\varphi_i^{p_2}(j)=a_j$ for $j\in\{0,1\}$ and that is undefined otherwise. We continue the construction inductively
and obtain computable sequences $(p_n)_{n\in\IN}$ in $\IN^\IN$ and $(w_n)_{n\in\IN}$ in $\IN^*$ in this way.
The sequence $(p_n)_{n\in\IN}$ converges to a computable $p\in\IN^\IN$.
This is because the construction above only depends on the computable function $h$ and yields longer and longer portions of $p$.
For every $d\in\{0,1\}^\IN$ we denote by $s_d$ the sequence $s_d:=d_0w_0d_1w_1d_2w_2...$ with $d_j:=d(j)$.
The construction ensures that $H\langle ip,s_d\rangle(n)\leq\varphi_i^p(n)$ for every $n\in\IN$ and $d\in\{0,1\}^\IN$
and hence $H\langle\langle i,p\rangle,s_d\rangle\not\in\NMAJ\langle i,p\rangle$ for every $d\in\{0,1\}^\IN$. On the other hand,
there is some $d\in\{0,1\}^\IN$ of hyperimmune degree, which implies $s_d\in\HYP^\DD K\langle i,p\rangle$,
since $d\leqT s_d$.
This yields a contradiction to the assumption and hence $\HYP^\DD\lW\HYP$.
\end{proof}

Hence, there is no uniform computable method
to find a hyperimmune $q\in\IN^\IN$ from an arbitrary member of a hyperimmune degree.
This also yields a second proof of $\NMAJ\nleqW\DIS$.
Finally, we mention that the separation in Proposition~\ref{prop:HYP} also yields a separation of the 
corresponding problems (i.e., the sets given by the respective solutions on computable instances) in the Medvedev lattice (see \cite[Theorem~9.1]{BGP21}).

We now want to show that $\MEET$ is equivalent to $\HYP$. The corresponding proof of \cite[Theorem~38]{HRSZ17} can be transferred into our setting.
On the first sight it might be a bit surprising that replacing $\not=$ by $\not\geq$ makes a difference, while replacing 
$<$ by $=$ does not. However, this comparison does not take the aspect of totality into account.
For completeness we include the proof, which is interesting by itself.

\begin{prop}
\label{prop:meet}
$\HYP\equivW\MEET$.
\end{prop}
\begin{proof}
For $p\in\IN^\IN$ we have
\[\MEET(p)=\{q\in\IN^\IN:(\forall i)(\varphi_i^p\mbox{ total }\TO(\exists n)\;\varphi_i^p(n)= q(n))\}.\]
We note that in the case of $\HYP$ we obtain a corresponding formulation with $\not\geq$ instead of $=$. In this case totality
does not need to be mentioned as the negative condition is automatically satisfied by partial $\varphi_i^p$.
It is obvious that $\HYP\leqSW\MEET$, as we just have to use $q\mapsto q+1$ to translate the solution on the same input $p$.
For the opposite direction we note that by the smn-theorem there is a computable $s:\IN\to\IN$ such that $\varphi_{s(i)}^p(n)=\Phi_{i}^p\langle i,n\rangle$ for all $i,n\in\IN$. 
Here $\Phi_i^p(n)$ denotes the time complexity, i.e., the number of computation steps required to compute $\varphi_i^p(n)$ (if it exists and undefined otherwise).
Then given an input $p\in\IN^\IN$ and $q\in\HYP(p)$ we can compute $r\in\IN^\IN$ with
\[r\langle i,n\rangle:=\left\{\begin{array}{ll}
  \varphi^p_i\langle i,n\rangle & \mbox{if $\Phi_i^p\langle i,n\rangle<q(n)$}\\
  0 & \mbox{otherwise}
\end{array}\right.\]
for all $i,n\in\IN$.
If $i\in\IN$ is such that $\varphi_i^p$ is total, then $\Phi_i^p$ and hence $\varphi^p_{s(i)}$ are total too and hence there is some 
$n\in\IN$ with $\Phi_i^p\langle i,n\rangle=\varphi_{s(i)}^p(n)<q(n)$.
This implies that $r\langle i,n\rangle=\varphi_i^p\langle i,n\rangle$, i.e., $r\in\MEET(p)$. 
Since $r$ can be computed, given $p,q$, we obtain $\MEET\leqW\HYP$.
\end{proof}

We note that the backwards reduction is not a strong Weihrauch reduction.

\section{Retractions, All-or-Unique Choice and G\"odel Numbers}
\label{sec:RET}

In the previous sections we have discussed the stashing of a number of problems of type $f:\In X\mto\IN$ with natural number output.
In this particular situation we can also describe stashing in an alternative way. 
This is because the space $\overline{\IN}^\IN$ is related to the space $\IN^{\In\IN}$ of partial functions $f:\In\IN\to\IN$
that we can represent by $\delta_{\IN^{\In\IN}}\langle i,p\rangle:=\varphi_i^p$ for all $i\in\IN,p\in\IN^\IN$. 
The exact relation between these two spaces is captured in the following lemma. The function $\iota$ essentially identifies $\uparrow$ (i.e., undefined) with $\bot$.

\begin{lem}[Space of partial functions]
\label{lem:partial}
The function $\iota:\IN^{\In\IN}\to\overline{\IN}^\IN$ with
\[\iota(p)(n):=\left\{\begin{array}{ll}
p(n) & \mbox{if $n\in\dom(p)$}\\
\bot & \mbox{otherwise}
\end{array}\right.\]
is computable and there is a computable function $\sigma:\overline{\IN}^\IN\to\IN^{\In\IN}$ such that
$\sigma\circ\iota(p)$ is an extension of $p$ for all $p\in\IN^{\In\IN}$.
\end{lem}

The proof is straightforward. We just note that a prefix of a name of $\bot$ can start like a name of a natural number $n\in\IN$
and continue with dummy symbols $0$, which means that it is actually a name of $\bot$. 
Hence, we do not get that the spaces $\overline{\IN}^\IN$ and $\IN^{\In\IN}$ are computably isomorphic.
However, Lemma~\ref{lem:partial} roughly speaking states that they are ``isomorphic up to extensions''.
Hence, for every property that is invariant under extensions, it does not matter whether we work with $\overline{\IN}^\IN$ or $\IN^{\In\IN}$.
This does, in particular, apply to stashing. Hence, for problems of type $f:\In X\mto\IN$ we can also describe stashing by the following definition.

\begin{defi}
For every problem $f:\In X\mto\IN$ we define $^\varphi f:\In X^\IN\mto\IN^\IN$ by
\[^\varphi f(x_n)_{n\in\IN}:=\{\langle i,p\rangle\in\IN^\IN:(\exists n\in\dom(\varphi_i^p))\;\varphi_i^p(n)\in f(x_n)\}\]
for all $(x_n)_{n\in\IN}\in\dom(f)^\IN$.
\end{defi}

Hence, as an immediate corollary of Lemma~\ref{lem:partial} we obtain the following corollary.

\begin{cor}[Stashing for discrete outputs]
\label{cor:discrete-stashing}
$\Sigma f\equivSW {\,}^\varphi f$ for all problems $f:\In X\mto\IN$.
\end{cor}

This again sheds  light on the fact that stashing of parallelizable problems gives us computability-theoretic problems.

We note that the usage of the G\"odel numbering and partial functions $\varphi_i^p$ in the previous sections was
mostly related to the input side of problems. Hence, implicitly, we have worked with the input space $\IN^{\In\IN}$.
Now, we want to get some better understanding of the effect of the space $\overline{\IN}^\IN$ on the output side.
Corollary~\ref{cor:discrete-stashing} can be seen as a way of replacing $\overline{\IN}^\IN$ on the output side by $\IN^\IN$.

For some problems, as those discussed in Propositions~\ref{prop:Sigma-ACC} and \ref{prop:Sigma-LPO} it appeared
that directly replacing $\overline{\IN}^\IN$ by $\IN^\IN$ on the output side was also possible without changing the degree.
However, the problem $\NGEQ$ from Definition~\ref{def:HYP} is an example that shows that this is not always possible.
If we just replace $\overline{\IN}^\IN$ by $\IN^\IN$ in $\Sigma(\NGEQ)$, then we obtain $\NMAJ$, which is not
equivalent to $\Sigma(\NGEQ)$ by Proposition~\ref{prop:NMAJ}. 

Hence, it is useful to have upper bounds on the price that such a direct replacement of $\overline{\IN}^\IN$ by $\IN^\IN$
costs. This is exactly captured by the following {\em retraction problem} $\RET_X$ that we define together with the closely related extension problem $\EXT_X$.

\begin{defi}[Retraction and extension problems]
\label{def:RET}
For every represented space $X$ and $Y\In\IN$ we define the following problems:
\begin{enumerate}
\item $\RET_X:\overline{X}\mto X,x\mapsto\{y\in X:x\in X\TO x=y\}$,
\item $\EXT_Y:\In\IN^\IN\mto Y^\IN,\langle i,p\rangle\mapsto\{q\in Y^\IN:q\mbox{ is a total extension of }\varphi_i^p\}$,\\
where $\dom(\EXT_Y):=\{\langle i,p\rangle:\range(\varphi_i^p)\In Y\}$.
\end{enumerate}
\end{defi}

The parallelization $\widehat{\RET_X}$ captures the complexity of translating $\overline{X}^\IN$ into $X^\IN$.
By Lemma~\ref{lem:partial} we obtain the following.

\begin{cor}
\label{cor:par-RET}
$\widehat{\RET_X}\equivSW\EXT_X$ for all $X\In\IN$.
\end{cor}

The problem $\RET_X:\overline{X}\mto X$ is a multi-valued retraction, i.e., a problem that satisfies
$\RET_X|_{X}=\id_{X}$. And it is the simplest such retraction in terms of strong Weihrauch reducibility.
Hence, we obtain the following result.

\begin{prop}[Multi-retraceability]
\label{prop:multi-retraceability}
Let $X$ be a represented space. Then there exists a computable multi-valued retraction $R:\overline{X}\mto X$ if and only if 
$\RET_X$ is computable. 
\end{prop}
\begin{proof}
It is clear that $\RET_X$ is a multi-valued retraction, which yields the ``if''--direction of the proof.
On the other hand, if there is a computable multi-valued retraction $R$, then $\RET_X\leqSW R$ and hence
$\RET_X$ is computable too. 
\end{proof}

Spaces that allow for computable multi-valued retractions were called {\em multi-retraceable} in \cite{BG20}.
This condition was further studied by Hoyrup in \cite{Hoy21} and related to fixed-point properties.
Here we rather have to deal with spaces that are not multi-retraceable.
In these cases we can consider the complexity of $\RET_X$ as a measure of how far the space $X$
is away from multi-retraceability or, in other words, how difficult it is to determine total extensions.
We gain some upper bounds directly from \cite[Proposition~2.10]{BG21a}.

\begin{fact}
\label{fact:RET}
$\RET_X\leqSW\lim$ for all totally represented $X$ and $\RET_\IN\leqSW\C_\IN$.
\end{fact}

For finite $X\In\IN$ we can classify $\RET_X$ somewhat more precisely. 
It is quite easy to see that the retraction problem $\RET_n$ is just equivalent to all-or-unique choice $\AoUC_n$
and hence located in between $\LLPO=\AoUC_2$ and $\C_n$, as well as $\LPO$.

\begin{prop}[Retraction problem]
\label{prop:RET}
$\LLPO\leqSW\RET_n\equivSW\AoUC_n\leqSW\C_n$, and\linebreak $\RET_n\leqSW\RET_\IN\equivW\LPO$ for all $n\geq2$.
\end{prop}
\begin{proof}
We start with proving $\RET_n\leqSW\AoUC_n$. Given a name $p$ of $x\in\overline{\IN}$ we generate
a name $q$ of the set $\{0,1,...,n-1\}$, as long as only dummy information appears in $p$.
If some prefix of $p$ starts to look like a name of some $k< n$, then we modify the output
$q$ to a name of the set $\{k\}$.  Since $k\in\RET_n(x)$, this yields the desired reduction.

Now we consider the inverse reduction $\AoUC_n\leqSW\RET_n$.
Given a name $p$ of a set $A\In\{0,1,...,n-1\}$ with $A=\{0,1,...,n-1\}$ or $A=\{k\}$ for some $k<n$,
we generate a name of $\bot$ as long as the name looks like a name of $\{0,1,...,n-1\}$. In the moment
where the represented set is clearly smaller, we wait until the information suffices to identify the singleton $\{k\}$,
and then we modify the output to an output of $k\in\overline{\{0,1,...,n-1\}}$.
If we apply $\RET_n$ to the generated point in $\overline{\{0,1,...,n-1\}}$ it yields a point in $A$.
This establishes the desired reduction.

The reductions $\LLPO=\AoUC_2\leqSW\AoUC_n\leqSW\C_n$ are clear for all $n\geq2$.
The reduction $\RET_n\leqSW\RET_\IN$ is easy to see, as well as $\RET_\IN\leqW\LPO$.
For the latter reduction, we consider a name $p$ of an input $x\in\overline{\IN}$. We use $\LPO$
in order to decide whether $p\not=\widehat{0}$. If $p=\widehat{0}$, then we can produce any
output $k\in\IN$; otherwise we search for the first non-zero component $k+1$ in $p$ and
produce the corresponding $k$ as output. This yields the reduction $\RET_\IN\leqW\LPO$.
Finally, the reduction $\LPO\leqW\RET_\IN$ can be seen as follows. Given an input $p\in\IN^\IN$,
we seek the first non-zero component. If this component appears in position $n$, then we generate
the output $x=n\in\overline{\IN}$. If there is no non-zero component, then we generate the output $x=\bot\in\overline{\IN}$.
Now given some $k\in\RET_\IN(x)$ and $p$, we search a non-zero component in $p$ up to position $k$.
If such a non-zero component appears, then $\LPO(p)=0$; otherwise $\LPO(p)=1$. This describes
the reduction $\LPO\leqW\RET_\IN$.
\end{proof}

Using this result, we obtain the pentagon of $\RET_n$ in the Weihrauch lattice.
By Fact~\ref{fact:parallelization-problem} we have $\widehat{\LLPO}\equivSW\widehat{\C_n}\equivSW\WKL$,
which implies $\widehat{\RET_n}\equivSW\WKL$.
On the other hand, by Proposition~\ref{prop:Sigma-LPO} we have $\uwidehat{\LLPO}\equivSW\uwidehat{\LPO}\equivSW\DIS$
and this yields $\RET_n\equivW\DIS$.

\begin{cor}[Pentagon of retractions]
\label{cor:RET}
$\uwidehat{\RET_n}\equivW\uwidehat{\RET_\IN}\equivW\DIS$, $\widehat{\RET_n}\equivSW\WKL$, and $\widehat{\RET_\IN}\equivW\lim$  for all $n\geq2$.
\end{cor}

This means that in the Weihrauch lattice  the pentagon of $\RET_n$ for $n\geq 2$ looks like the pentagon of $\LLPO$
in Figure~\ref{fig:LLPO-pentagon}, while the pentagon of $\RET_\IN$ looks like the pentagon of $\LPO$
in Figure~\ref{fig:LPO-pentagon} (except that
one cannot use strong Weihrauch reducibility). The first mentioned fact is also interesting, as the problems $\RET_n$ form a strictly increasing chain, as we prove next by a simple cardinality argument.

\begin{prop}
$\AoUC_n\lW\AoUC_{n+1}\lW\AoUC_\IN$ for all $n\geq2$.
\end{prop}
\begin{proof}
It suffices to prove $\AoUC_{n+1}\nleqW\AoUC_n$ for all $n\geq2$.
Let us assume for a contradiction that $\AoUC_{n+1}\leqW\AoUC_n$ for some $n\geq 2$ via computable $H,K$.
Now consider the name $p=000...$ of the empty set $\emptyset=\range(p-1)$. Then $a_k:=H\langle p,k\rangle\in\{0,...,n\}$ for $k<n$ satisfy $|\{a_0,...,a_{n-1}\}|\leq n$.
For each $k<n$ there is some prefix of $p$ of length $l_k$ that suffices to produce the output $a_k$. For the prefix $l:=\max\{l_0,...,l_n\}$ we have
\[|H\langle p|_l\IN^\IN\times\{0,...,n\}\rangle|=|\{a_0,...,a_{n-1}\}|\leq n,\]
despite the fact that $p|_l\IN^\IN$  contains names of all singleton sets $A\In\{0,1,...,n\}$.
Since there are $n+1$ such sets, this is a contradiction!
\end{proof}

That means that with the problems $\RET_n\equivSW\AoUC_n$ we have a strictly increasing sequence of problems
in between $\LLPO$ and $\LPO$ that all parallelize and stash to the same problems, respectively. 
This is a strong refutation of injectivity of the following operation.

\begin{cor}
\label{cor:injectivity}
The map $f\mapsto(\uwidehat{f},\widehat{f})$ is not injective on Weihrauch degrees.
\end{cor}

This means that pentagons do not characterize the Weihrauch degrees that generate them, even though they reveal some useful
information about them.

\section{Limit Avoidance and a Phase Transition}
\label{sec:lim}

The pentagons determined so far are all related to the discontinuity problem.
This could lead to the false conclusion that $\DIS$ is the bottom vertex of typical pentagons.
However, this is rather based on the fact that our study of pentagons was motivated
by an analysis of the discontinuity problem. 
In this closing section we study the weakest problem known to us
whose stashing is not $\DIS$.
This is the {\em limit avoidance problem} $\NLIM_\IN$, defined below.
We define it together with the related {\em non-lowness problem} $\NLOW$.
  
\begin{defi}[Problems related to limits]
We consider the following problems for $X\In\IN$:
\begin{enumerate}
\item $\NLIM_X:\In X^\IN\mto X,p\mapsto\{k\in X:\lim_{n\to\infty}p(n)\not=k\}$ with\\
        $\dom(\NLIM_X)=\{p:(p(n))_{n\in\IN}$ converges$\}$,
\item $\NLOW:\DD\mto\DD,a\mapsto\{b\in\DD:b'\not\leq a'\}$.
\end{enumerate}
\end{defi}

The following fact is easy to see:

\begin{fact}
\label{fact:NLIM}
$\ACC_\IN\leqSW\NLIM_\IN\leqSW\NLIM_2\equivSW\lim_2\leqSW\lim_\IN\equivSW\C_\IN$.
\end{fact}

If we can prove that $\DIS\lW\uwidehat{\NLIM_\IN}$ holds, then this implies that the stashings of all problems
above $\NLIM_\IN$ also lie above $\DIS$. We prove the following stronger result.

\begin{prop}
\label{prop:NLIM-WKL}
$\uwidehat{\NLIM_\IN}\nleqW\WKL$.
\end{prop}
\begin{proof}
We recall that we are working with
\[\uwidehat{\NLIM_\IN}:\In\IN^\IN\mto\overline{\IN}^\IN,\langle p_0,p_1,p_2,...\rangle\mapsto\{q\in\overline{\IN}^\IN:(\exists n)\;\lim_{i\to\infty}p_i(n)\not=q(n)\in\IN\}.\]
Let us assume for a contradiction that $\uwidehat{\NLIM_\IN}\leqW\WKL$ holds. As $\WKL$ is a cylinder (i.e., satisfies $\id\times\WKL\equivSW\WKL$),
it follows that $\uwidehat{\NLIM_\IN}\leqSW\WKL$. Since stashing is an interior operator, this yields $\uwidehat{\NLIM_\IN}\leqSW\WKL^\DD\equivSW\PA$ by Fact~\ref{fact:PA}.
But this means that every fixed $r\in2^\IN$ of PA--degree has the property that it computes some fixed $q$ with $q\in\uwidehat{\NLIM_\IN}\langle p_0,p_1,p_2,...\rangle$
for all computable $\langle p_0,p_1,p_2,...\rangle\in\dom(\uwidehat{\NLIM_\IN})$.
This is because the reduction $\uwidehat{\NLIM_\IN}\leqSW\PA$ is a strong one and hence $q$ only depends on $r$.
There are low $r\in 2^\IN$ of PA--degree, i.e., such that $r'\leqT\emptyset'$. This is because the set of diagonally non-computable binary functions is co-c.e.\ closed, every such function is of 
PA-degree (see \cite[Theorem~10.3.3]{Soa16}) and by the low basis theorem there is a low function among those (see \cite[Theorem~3.7.2]{Soa16}).
Hence there is some fixed low $q$ with $q\in\uwidehat{\NLIM_\IN}\langle p_0,p_1,p_2,...\rangle$
for all computable $\langle p_0,p_1,p_2,...\rangle\in\dom(\uwidehat{\NLIM_\IN})$. 
By Corollary~\ref{cor:RET} this implies that there is some $t\in\widehat{\RET_\IN}(q)$, which is limit computable, as $\widehat{\RET_\IN}\leqSW\lim$ and limit computable
problems yield some limit computable outputs on low inputs. This means that for a fixed limit computable $t:\IN\to\IN$ we have that
$\lim_{i\to\infty}p_i\not= t$
for all computable $\langle p_0,p_1,p_2,...\rangle\in\dom(\uwidehat{\NLIM_\IN})$, which is impossible, as $\lim_{i\to\infty}p_i$ is limit computable.
\end{proof}

We mention that Fact~\ref{fact:NLIM} implies that $\DNC_\IN\leqSW\widehat{\NLIM_\IN}\leqSW\lim$. 
We leave it to a future study to determine the exact pentagons of the problems listed in Fact~\ref{fact:NLIM} other than $\ACC_\IN$.
Here we just mention one upper bound on $\Pi\Sigma(\NLIM_\IN)$.

Since $\DIS\lW\NON\lW\WKL$ it is clear that Proposition~\ref{prop:NLIM-WKL} implies $\uwidehat{\NLIM_\IN}\nleqW\DIS$
and also $\Pi\Sigma(\NLIM_\IN)\nleqW\NON$. It is not difficult to show that the non-lowness problem is an upper bound
for $\Pi\Sigma(\NLIM_\IN)$.

\begin{prop}
\label{prop:NLOW}
$\Pi\Sigma(\NLIM_\IN)\leqSW\NLOW$.
\end{prop}
\begin{proof}
For simplicity we work with the equivalent definition 
\[\NLOW:\IN^\IN\mto\IN^\IN,p\mapsto\{q\in\IN^\IN:q'\nleqT p'\}.\]
We first note that $\NLOW$ is parallelizable: this follows since $q'\nleqT\langle p_0,p_1,p_2,...\rangle'$ implies $q'\nleqT p_i'$ for all $i\in\IN$.
Hence, it suffices to show $\uwidehat{\NLIM_\IN}\leqSW\NLOW$ in order to obtain $\Pi\Sigma(\NLIM_\IN)\leqSW\NLOW$.
Now given $p:=\langle p_0,p_1,p_2,...\rangle$ such that $r=\lim_{i\to\infty}p_i$ exists in $\IN^\IN$, we have $r\leqT p'$.
Now let $q\in\IN^\IN$ be such that $q'\nleqT p'$. Then using the time complexity function $\Phi_i^q$ we can define 
by $s\langle i,n\rangle:=\Phi^q_i(n)$ a partial function $s:\In\IN\to\IN$
that is computable from $q$. This function has the property that every total extension $t:\IN\to\IN$ of it computes $q'$.
This is because we can simulate the computation of $\varphi^q_i(n)$ for given $i,n$ for $t\langle i,n\rangle$ time steps. And
either the computation halts within this time bound, which implies $q'\langle i,n\rangle=1$ or it does not halt, which implies that $t\langle i,n\rangle$ is not a 
correct time bound, hence $\Phi^q_i(n)$ is undefined and hence $q'\langle i,n\rangle=0$.
By Lemma~\ref{lem:partial} we can consider $s$ as a function of type $s:\overline{\IN}\to\IN$. 
Since every total extension $t:\IN\to\IN$ of it computes $q'$, we obtain $t\nleqT p'$, which implies $r\not=t$ for every such extension $t$.
Hence there is some $n\in\IN$ with $r(n)\not=s(n)\in\IN$, i.e., $s\in\uwidehat{\NLIM_\IN}(p)$. 
This establishes the reduction $\Pi\Sigma(\NLIM_\IN)\leqSW\NLOW$.
\end{proof}

\section{Conclusions}
\label{sec:conclusions}

We have introduced the stashing operation as a dual of parallelization and we have proved that it is 
an interior operator. The action of parallelization and stashing on Weihrauch degrees naturally leads
to the study of stashing-parallelization pentagons, which can be used to describe a number of
natural Weihrauch degrees. In many cases of the studied pentagons the discontinuity problem
featured as the bottom problem of the respective pentagon.
However, we were also able to identify a phase transition point, where this no longer
happens to be the case.
The duality inherent in pentagons needs to be studied further in order to simplify
the calculation of pentagons, which tends to be difficult in more advanced examples.

\bibliographystyle{alpha}
\bibliography{lit}

\section*{Acknowledgments}

This work has been supported by the {\em National Research Foundation of South Africa} (Grant Number 115269).

\end{document}